\documentclass[11pt,a4paper,english]{article}
\usepackage[sort,numbers]{natbib}

\pdfoutput=1

\usepackage{array, verbatim}
\usepackage{amsfonts}
\usepackage{amssymb}
\usepackage{amsthm}
\usepackage{mathtools}
\usepackage{mathtext}
\usepackage{mathrsfs}
\usepackage{dsfont}
\usepackage{xcolor}
\usepackage{todonotes}
\usepackage[utf8]{inputenc}
\usepackage[english]{babel}
\usepackage[affil-it]{authblk}

\usepackage{makeidx}
\usepackage{float}
\usepackage[hyphens]{url}
\usepackage{enumerate}
\usepackage{hyperref}

\hypersetup{
    colorlinks=false,
    pdfborder={0 0 0},
}

\newcommand\RMA[1]{\textcolor{red}{#1}}

\DeclareMathOperator*{\esssup}{ess\,sup}

\DeclareMathOperator*{\argmax}{arg\,max}

\def\R{\mathbb{R}}

\def\bl{\textcolor{blue}}

\newtheorem{theorem}{Theorem}[section]
\newtheorem{lemma}[theorem]{Lemma}
\newtheorem{corollary}[theorem]{Corollary}
\newtheorem{proposition}[theorem]{Proposition}

\theoremstyle{definition}
\newtheorem{definition}{Definition}
\newtheorem{assumption}{Assumption}
\theoremstyle{remark}
\newtheorem{remark}[theorem]{Remark}

\allowdisplaybreaks[1]

\makeatletter  % so @ will be a letter
\@addtoreset{equation}{section}

\makeatother  % so  @ will not be a letter
\begin{document}
	\title{\textbf{Nonzero-sum optimal stopping games and generalised Nash equilibrium}}
	
	\author[1]{Randall Martyr\thanks{Corresponding author. Email: r.martyr@qmul.ac.uk}\thanks{Financial support received from the EPSRC via grant EP/N013492/1.}}
	\author[1]{John Moriarty\thanks{Financial support received from the EPSRC via grant EP/K00557X/2.}}
	\affil[1]{School of Mathematical Sciences, Queen Mary University of London, Mile End Road, London E1 4NS, United Kingdom.}
	\maketitle
	\begin{abstract}
		In the nonzero-sum setting, we establish a connection between Nash equilibria in games of optimal stopping (Dynkin games) and generalised Nash equilibrium problems (GNEP). In the Dynkin game this reveals novel equilibria of threshold type and of more complex types, and leads to novel uniqueness and stability results.
%The game is played on an infinite time horizon and the underlying randomness is given by a time-homogeneous strong Feller process on a Hausdorff, locally compact, countable base (LCCB) space \ldots After studying carefully the optimal stopping problems faced by each player, we specialise to the case of a Brownian motion on $[0,1]$, with $\{0,1\}$ either an absorbing or killing boundary. 
	\end{abstract}

%%%%%%%%%%%%%%%%%%%%%%%%%%%%%%%%%%%%%%%%%%%%%%%%%%%%%%%%%%%%%%%%%%%%%%%%%%%%%%%%%%%%%%%%%%%%%%%%%%%%%%%%%%%%%%%%%%%%%%%%%%%%%%%%%%%%%%%%%%%%%%%%%%%%%%%%%%%%%%%%%%%%%%%%%%%%%%%%%%%%%%%%%%%%%%%%%%%%%%%%%%%%%%%%%%%%%%%%%%%%%%%%%%%%%%%%%%%%%%%%%%%%%%%%%%%%%%%%%%%%%%%%%%%%%%%%%%%%%%%%%%%%%%%%%%%%%%%%%%%%%%%%

\section{Introduction}

In this paper we establish a connection between Nash equilibria in two different types of game. The first type is the two-player, nonzero-sum Dynkin game of optimal stopping (for general background on optimal stopping problems the reader is referred to \cite{Peskir2006}).
Player $i \in \{1,2\}$ chooses a stopping time $\tau_i$ for a 
strong Markov process %one-dimensional standard Brownian motion $W=(W_t)_{t \geq 0}$
$X=(X_t)_{t \geq 0}$ defined on the interval $(x_\ell, x_r)$. {\em Reward functions} $f_i$, $g_i$, $h_i$ are given and the reward or {\em payoff} to player $i$ is
\begin{equation}\label{eq:Game-Payoff-Functional}
\begin{split}
\mathcal{J}_{i}(\tau_{1},\tau_{2}) \coloneqq f_{i}(X_{\tau_{i}})\mathds{1}_{\{\tau_{i} < \tau_{-i}\}} + g_{i}(X_{\tau_{-i}})\mathds{1}_{\{\tau_{-i} < \tau_{i}\}} + h_{i}(X_{\tau_{i}})\mathds{1}_{\{\tau_{i} = \tau_{-i}\}},%, \quad j = 3-i.
\end{split}
\end{equation}
%Here $\lambda \geq 0$ and f
where for each player $i \in \{1,2\}$ the subscript $-i$ denotes the other player. In this context equilibrium strategies $(\tau_1, \tau_2)$ of the form
\begin{equation}\label{eq:thresholdform}
\tau_1 = \inf\{t \geq 0: X_t \leq \ell\} \quad \text{ and } \quad \tau_2 = \inf\{t \geq 0: X_t \geq r\},
\end{equation}
for constants $\ell, r \in (x_\ell,x_r)$ with $\ell < r$, 
are referred to as {\em threshold-type} equilibria. A recent example is in  \cite{DeAngelis2015}, in which the thresholds $\ell$, $r$ are drawn from the disjoint {\em strategy spaces} $\mathcal S_1$ and $\mathcal S_2$ respectively where
\begin{equation}\label{eq:feassets}
\mathcal S_1:=[x_\ell,a], \qquad \mathcal S_2:=[b,x_r],
\end{equation} 
for some constants $a,b$ with $x_\ell<a<b<x_r$. 

The second type of game is a deterministic {\em generalised game} \cite{Facchinei2007} (or \emph{abstract economy} \cite{Arrow1954}) with $n \geq 2$ players, where $n$ will depend on the structure of the equilibrium studied in the Dynkin game. Since the examination of all cases $n \geq 2$ is reserved for future work, however, we focus on $n=2$ and simply provide an example with $n=3$. 

%JM: Now removed, since it's almost a repetition of the above:
%In both games we take the players' strategy spaces to be 
%$$\mathcal S_1:=[0,a], \qquad \mathcal S_2:=[b,1], \qquad \text{ for some constants } a,b \in (0,1).$$

The connection yields novel equilibria in the Dynkin game. This novelty is threefold. Firstly the reward functions are not required to be differentiable. Secondly we obtain novel equilibria of threshold type, since both cases $a<b$ and $a \geq b$ are permitted. Thirdly, while threshold-type equilibria correspond to the case $n=2$, the cases $n>2$ yield equilibria with more complex structures. To the best of our knowledge, these complex equilibrium structures in the Dynkin game have not been previously studied.

In the threshold-type case, we obtain the uniqueness of the equilibria among Markovian strategies, and results about their local and global stability.

\subsection{Background}

The structure of Nash equilibria in nonzero-sum Dynkin games has recently been investigated in \cite{Attard2015} and \cite{DeAngelis2015}, where sufficient conditions for the existence and uniqueness of threshold-type equilibria are obtained. A key difference between the case $n=2$ of the present paper and the latter work is that there, the functions $f_i$ in \eqref{eq:Game-Payoff-Functional} are twice differentiable and have unique points of inflexion $a$ and $b$ respectively with $a < b$, conditions which may all be relaxed in the present approach. %In order to place our work further in the context of \cite{DeAngelis2015}, 
Appendix \ref{rem:ext} contains remarks on the inclusion of time discounting, and the use of other Markov processes $X$, in our setup.

Our results on stability relate to an iterative approximation scheme for Nash equilibria, 
which has been previously studied outside the Markovian framework in \cite{Hamadene2010} and, in the Markovian framework, in \cite{BensoussanFriedman1977}, \cite{Cattiaux1990}, \cite{KARATZASSudderth2006} and \cite{Nagai1987}. In \cite{KARATZASSudderth2006} it is assumed that $f_{i} = g_{i}$ and in \cite{BensoussanFriedman1977}, \cite{Cattiaux1990} and \cite{Nagai1987} a condition related to %$\lambda$-
superharmonicity is imposed for the $g_{i}$.
% (see Definition~\ref{Definition:Superharmonic-Function} below). 
The latter conditions ensure monotone convergence over the iteration, whereas the approach via stability in Section \ref{Section:Stability-Uniqueness} does not rely on monotonicity.

The special case of zero-sum Dynkin games, in which $f_i=-g_{-i}$ and $h_i=-h_{-i}$, has received particular attention in the literature. Thorough analyses of the zero-sum game for a large class of driving Markov processes can be found in \cite{Ekstrom2008} and \cite{Peskir2009} and in that context Assumption \ref{assumption:Preliminary-Assumption-1}, which relates the game to a {\em war of attrition} \cite[Section 4.5.2]{Fudenberg1991}, is sufficient for the existence of a Nash equilibrium among pure strategies. We adopt the same setting, as is also common in the nonzero-sum context (see for example \cite{BensoussanFriedman1977,Cattiaux1990,Nagai1987}).
\begin{assumption}\label{assumption:Preliminary-Assumption-1}
	For $i = 1,2$ the functions $f_{i}$, $g_{i}$ and $h_{i}$ are bounded and continuous on $[x_\ell, x_r]$, and satisfy $f_{i} \le h_{i} \le g_{i}$.
\end{assumption}
%JM: On reflection I wonder whether we could perhaps just leave it to the referees to pick up on the following point if they want to?:
%We note that in \cite{Attard2015} it is assumed that $f_{i}(x) = g_{i}(x)$ at the boundaries $x \in \{x_\ell, x_r\}$: as we shall see in Section \ref{sec:pre} below, without loss of generality these boundary values may be assumed to be zero in our setup.
%Although it is not known whether 
%In addition to Assumption \ref{assumption:Preliminary-Assumption-1},  
%is sufficient for the existence of equilibria in the Markovian nonzero-sum setting, it is nevertheless  
%Prior to the papers \cite{Attard2015,DeAngelis2015}, 
%\noindent The conditions imposed beyond Assumption \ref{assumption:Preliminary-Assumption-1} in the present work will be stated as Assumption \ref{Assumption:Infimum-Approach} and Condition U. 
In Remark~\ref{Remark:Waiting-Incentive} we discuss how Assumption~\ref{assumption:Preliminary-Assumption-1} can be weakened without affecting the main results.

%JM: We don't want to seem `parochial' by mentioning the following papers more than necessary:
%On the other hand, the papers \cite{Attard2015,DeAngelis2015} proved the existence of Nash equilibria in the Dynkin game without those strong assumptions on $g_{i}$, imposing conditions on $f_{i}$ that enforce a particular structure of the Nash equilibria instead. As mentioned previously, this paper extends the study of Nash equilibria in \cite{Attard2015,DeAngelis2015} by relaxing the strong conditions they used to obtain their results.

%In order to link the two game settings above we first make the following definition.

\subsection{Preliminaries}\label{sec:setting}

In this section we recall necessary background on subprocesses, superharmonic %, supermedian and excessive
and quasi-concave functions, %and also define the {\em reduced functions}
which should be familiar.

\subsubsection{Subprocesses of a Brownian motion}\label{sec:subpr}
%let $E$ be a set and $\mathcal{B}(E)$ denote the Borel $\sigma$-algebra on $E$. We suppose that the space $(E,\mathcal{B}(E))$ is Hausdorff, locally compact and has a countable base (LCCB). Central to the discussion is

Let $W = (W_{t})_{t \ge 0}$ be a one-dimensional standard Brownian motion defined on a filtered probability space $(\Omega,\mathcal{F},\mathbb{F}=(\mathcal{F}_{t})_{t \geq 0},\hat{\mathds{P}})$, where $\mathbb{F}$ is the universally completed filtration \cite[p.~27]{Blumenthal1968}. We will write the probability measure as $\hat{\mathds{P}}^{x}$ in the case $\hat{\mathds{P}}(\{W_0=x\}) = 1$, and denote the expectation operator with respect to $\hat{\mathds{P}}^{x}$ by $\hat{\mathds{E}}^{x}$.
From $W$ we derive {\em subprocesses} in the sense of \cite[Chapter~\MakeUppercase{\romannumeral 3}]{Blumenthal1968}. More precisely, for each subset $E$ of $[0,1]$ we define a subprocess $X^E = X = (X_{t})_{t \ge 0}$ such that $X$ and its (almost surely finite) lifetime $\zeta$ satisfy,
\begin{eqnarray}
\zeta &=& \inf\{t \ge 0 \colon W_{t} \notin E\}, \\
X_t &=&
\begin{cases}
W_t, & 0 \leq t < \zeta, \\
\Delta, & t \geq \zeta.
\end{cases}
\end{eqnarray}
Here $\Delta$ is a cemetery state and the state space of $X$ is $E$ equipped with its Borel sigma-algebra $\mathcal{B}(E)$ and augmented by $\Delta$. We set $\phi(\Delta) = 0$ for every measurable function $\phi$ on $E_{\Delta} \coloneqq E \cup \Delta$, showing in Section \ref{sec:pre} below that this choice involves no loss of generality. For $x \in E$ let $\mathds{P}^{x}$ and $\mathds{E}^{x}$ denote the probability measure and expectation operator corresponding to $X$ when $\mathds{P}^{x}(\{X_{0} = x\}) = 1$. For every measurable function $\phi$ vanishing outside $E$ and every $t \ge 0$ we have \cite[p.~105]{Blumenthal1968}:
\begin{equation}\label{eq:Vanishing-After-Death}
\mathds{E}^{x}[\phi(X_{t})] = \hat{\mathds{E}}^{x}[\phi(W_{t})\mathds{1}_{\{t < \zeta\}}].
\end{equation}
%JM: I wasn't sure exactly what the following sentence meant:
%\RMA{Extending \eqref{eq:Vanishing-After-Death} to stopping times $\tau$ shows that these expectations are unchanged if $\tau$ is replaced by $\tau \wedge \zeta$.}
For each measurable set $A$ %\subseteq E$ 
%JM: By letting $A$ be a general measurable set, rather than restricting it to be a subset of $E$, we may avoid potential complications below.
we write the associated first entrance %or {\em debut}
 time of $X$ as 
\begin{equation}\label{eq:debut}
D_{A} \coloneqq \inf\{t \ge 0 \colon X_{t} \in A\} = \inf\{t > 0 \colon X_{t} \in A\} \qquad \text{ a.s., }
\end{equation}
where we take $\inf \emptyset = \zeta$ (the second equality follows since every point is regular for Brownian motion, see for example \cite[Remark 8.2]{morters2010brownian}).

\subsubsection{Superharmonic functions}
\label{sec:fnclasses}

We will use the fact that value functions of various optimal stopping problems for these subprocesses are superharmonic (see e.g. \cite{Dayanik2003} and Proposition \ref{Prop:Superaveraging-Property}). Let $E \subseteq [0,1]$, $A \in \mathcal{B}(E_{\Delta})$, and write $\mathcal{T}$ for the set of all $\mathbb{F}$-stopping times with values in $\R_+ \cup \{\infty\}$.

%JM: I don't think we use the following fact:
% (which are permitted to take the value $+\infty$ so that $\mathcal T$ is the set of {\em Markov times} in the older literature).

\begin{definition}\label{Definition:Superharmonic-Function}
	A measurable function $\phi \colon E_\Delta \to \mathbb{R}$ is said to be {\em 
	%$\lambda$-
	superharmonic} (resp. 
	%$\lambda$-
	harmonic) on $A$
	%a measurable set $A \subseteq E$ 
	if for every $x \in E$ and $\tau \in \mathcal{T}$:
	\[
	\phi(x) \ge\text{(resp. $ = $)}\,\, \mathds{E}^{x}[%e^{-\lambda (\tau \wedge D_{A^{c}})}
	\phi(X_{\tau \wedge D_{A^{c}}})].
	\]
A measurable function $\phi \colon E_\Delta \to \mathbb{R}$ is said to be {\em 
%$\lambda$-
subharmonic} on $A$ if $-\phi$ is 
%$\lambda$-
superharmonic on $A$, and {\em 
%$\lambda$-
harmonic} on $A$ if it is both 
%$\lambda$-
superharmonic and 
%$\lambda$-
subharmonic on $A$.
If $A = E$ then the term superharmonic, subharmonic, or harmonic is used as appropriate.
%JM: Not sure that the following are needed:
%, and {\em $\lambda$-harmonic} in $A$ if it is both $\lambda$-superharmonic and $\lambda$-subharmonic in $A$. If $A = E$ then the term $\lambda$-superharmonic is used as appropriate. 
%The qualifier `$\lambda$' is omitted when $\lambda = 0$.
\end{definition}

Using %an extension of
the strong Markov property, %via Galmarino's Test (Theorem \MakeUppercase{\romannumeral 4}.103 of Dellacherie and Meyer \cite{Dellacherie1978})
one can show that (see \cite[p.~561]{Peskir2009} for details): if $\phi$ is 
%$\lambda$-
superharmonic then
\begin{equation}\label{Definition:Supermedian-Function}
%e^{-\lambda \nu}
\phi(X_{\nu}) \ge \mathds{E}^{x}[
%e^{-\lambda \rho}
\phi(X_{\rho}) \vert \mathcal{F}_{\nu}], \enskip \text{a.s.}\,\,\forall\, \rho,\nu \in \mathcal{T} \,\,\text{such that}\,\, \nu \le \rho.
\end{equation}
In other words, $\phi$ is a 
%$\lambda$-
superharmonic function if and only if $(
%e^{-\lambda t}
\phi(X_{t}))_{t \ge 0}$ is a strong supermartingale with respect to $\mathbb{F}$. Here the qualifier `strong' refers to the extension of the supermartingale property to stopping times. In our setup the 
%$\lambda$-
superharmonic functions on $E$
%JM: Symbol was not used here: 
%$F$ 
are also equivalent to
%JM: Clarifying the logic here:
% non-negative superharmonic functions are 
the {\em strongly supermedian} functions on $E$ (see for example \cite{Cattiaux1990}, \cite{Karoui1992} and \cite{Mertens1973}), as follows.
%
%JM: I think this risks sounding a bit patronising for our target audience:
%, as in the relationship between the Markov and strong Markov properties. 
Taking $A = E$ and $\tau = \zeta$ in Definition \ref{Definition:Superharmonic-Function}, the convention $\phi(\Delta) = 0$ implies that the 
%$\lambda$-
superharmonic functions $\phi$ on $E$
%JM: Symbol was not used here: 
%$F$ 
are non-negative and are therefore strongly supermedian. Moreover, since $X$ is a subprocess of Brownian motion, superharmonic (respectively subharmonic and harmonic) functions are concave (resp. convex, linear) on convex subsets of $E$ (see %\cite[p.~28]{Borodin1996}
\cite[p.~179]{Dayanik2003}). %The {\em excessive} functions 
% are then the non-negative, continuous, concave functions.
%JM: Clarifying the logic here:
% non-negative superharmonic functions are 
%precisely the {\em strongly supermedian} functions $\phi$ on $E$ 
%JM: Is the following important (or used)?:
%The latter were defined originally to be the non-negative functions satisfying~\eqref{Definition:Supermedian-Function}. 

%JM: This fact doesn't seem to be used:
%A measurable function $\phi$ is said to be {\em finely continuous} if it is almost surely right-continuous along the paths of $X$ \cite{Blumenthal1968,Chung2005}. 

%precisely the non-negative finely continuous superharmonic functions. In the particular case of Brownian motion on $E$, continuity and fine continuity are equivalent and 

%the finely continuous superharmonic (respectively subharmonic and harmonic) functions are continuous and concave (resp. convex, linear) -- see \cite[p.~28]{Borodin1996}, \cite[p.~179]{Dayanik2003}. In fact, these two notions of continuity agree in this case \cite{Blumenthal1968,Chung2005}.

We will make repeated use of the following transformation.

\begin{definition}\label{def:red}
Given $A \in \mathcal{B}(E_{\Delta})$ 
%JM: Repetition:
%be given and let $D_{A}$ denote its debut time for $X$. For 
and a bounded measurable function $\phi \colon E_\Delta \to \mathbb{R}$, and recalling the first entrance time defined in \eqref{eq:debut}, define $\phi_{A} \colon E_\Delta \to \mathbb{R}$ by
	\begin{equation}\label{eq:Lambda-Reduced-Function}
	\phi_{A}(x) \coloneqq \mathds{E}^{x}\left[
	%e^{-\lambda D_{A}}
	\phi(X_{D_{A}})\right].
	\end{equation}
\end{definition}
%We will call the function $\phi_{A}$ the \emph{%$\lambda$-reduction (of $\phi$ on $A$)}.
It is not difficult to show (using the strong Markov property) that for any measurable function $\phi$, the function
$\phi_{A}$ is 
%$\lambda$-
harmonic on $A^{c}$, and is 
%$\lambda$-
superharmonic if $\phi$ is 
%$\lambda$-
superharmonic.
% Further when $X$ is a subprocess of Brownian motion, as in this paper, the function $x \mapsto \phi_{A}(x)$ 
Moreover, it is continuous whenever $\phi$ is continuous and $A$ is closed in $E$ \cite{Schilling2012}.
%}
%\todo[inline]{RM: I've used the concept of subprocess here again. Note that we can rewrite \eqref{eq:Lambda-Reduced-Function} in terms of the first exit time of $W$ from $A^{c}$, so that our statement agrees with the results in \cite{Schilling2012}} 
%Note that the reduced function has also been defined in the literature using the first hitting time $T_{A} = \inf\{t > 0 \colon X_{t} \in A\}$ in place of the debut time $D_A$ in \eqref{eq:Lambda-Reduced-Function}. 
%The two definitions clearly agree on the complementary set $A^{c}$, and furthermore they agree everywhere on $E$ when $A$ is finely open \todo{finely open in $E$?} \cite[p.~90]{Cattiaux1990}.

\subsubsection{Quasi-concavity}

For use in the existence results below, we recall the definition and some properties of quasi-concave functions (see e.g. \cite[Chapter 3.4]{Boyd2004}). The extended real line will be denoted by $\bar{\mathbb{R}} = [-\infty,+\infty]$. 

\begin{definition}\label{Definition:Quasi-Concave-Function}
	Let $\mathcal{D} \subseteq \mathbb{R}$ be convex. A function $F \colon \mathcal{D} \to \bar{\mathbb{R}}$ is said to be {\em quasi-concave} if for every $\alpha \in \mathbb{R}$ the superlevel sets $L^{+}_{\alpha}$ defined by
	\[
	L^{+}_{\alpha} = \left\{x \in \mathcal{D} \colon F(x) \ge \alpha \right\}
	\]
	are convex. If the same statement holds but with the sets $\left\{x \in \mathcal{D} \colon F(x) > \alpha \right\}$ then $F$ is said to be {\em strictly} quasi-concave. A function $F$ is said to be (strictly) quasi-convex on a convex domain $\mathcal{D}$ if and only if $-F$ is (strictly) quasi-concave.
\end{definition}
All concave functions are quasi-concave. Moreover a function $F \colon \mathcal{D} \to \bar{\mathbb{R}}$ is quasi-concave if and only if $\mathcal{D}$ is convex and for any $x_{1},x_{2} \in \mathcal{D}$ and $0 \le \theta \le 1$ we have
\begin{equation}\label{eq:Quasi-Concavity-Alt}
F(\theta x_{1} + (1-\theta)x_{2}) \ge \min(F(x_{1}),F(x_{2})).
\end{equation}
If \eqref{eq:Quasi-Concavity-Alt} holds with strict inequality then $F$ is strictly quasi-concave.

%\todo[inline]{RM: I have just defined superharmonic functions, gave the defining strong supermartingale property, and commented on the terminology `strongly supermedian functions'. I have also discussed the relationship between superharmonic functions and concave ones.}
\medskip
The remainder of this paper is organised as follows. In Section \ref{sec:games} the two game settings are presented and connected. Useful alternative expressions for the expected payoffs in the Dynkin game, as optimal stopping problems for subprocesses, are developed in Section \ref{sec:fktsp},
and our main 
%JM: hopefully uniqueness too!:
existence and uniqueness
results follow in Sections \ref{Section:Existence-Of-Equilibria} and \ref{Section:Stability-Uniqueness}. Finally, in Section~\ref{sec:examples} we present an extension for a more complex equilibrium structure.% and discuss recover some known equilibria as special cases, and

\section{Two games}\label{sec:games}
\subsection{Generalised Nash equilibrium}\label{sec:gnep}

In the $n$-player generalised game each player's set of available strategies, or {\em feasible strategy space}, depends on the strategies chosen by the other $n-1$ players. The case $n=2$ is as follows. 
Player $i \in \{1,2\}$ has a {\em strategy space} $\mathcal{S}_{i}$ and a set-valued map $K_{i} \colon \mathcal{S}_{-i} \rightrightarrows \mathcal{S}_{i}$ determining their {\em feasible} strategy space. Denoting a generic strategy for player $i$ by $s_i$, a strategy pair $(s_1,s_2)$ is then {\em feasible} if $s_i \in K_i(s_{-i})$ for $i=1,2$. Setting $\mathcal{S}_{1} = [0,a]$ and $\mathcal{S}_{2} = [b,1]$, the pair of mappings $K_{1} \colon [b,1] \rightrightarrows [0,a]$ and $K_{2} \colon [0,a] \rightrightarrows [b,1]$ will be given by
\begin{equation}
\label{eq:01}
\begin{split}
K_{1}(y) & = [0,y \wedge a], 
\\
K_{2}(x) & = [x \vee b,1], 
\end{split}
\end{equation}
where $a$ and $b$ are given constants lying in the interval $(0,1)$.
That is, the feasible strategy pairs are given by the convex, compact set
\begin{equation}\label{eq:constraintset}
\mathcal{C} = \{(x,y) \in [0,a] \times [b,1] \colon x \le y\}.
\end{equation}
This choice of $\mathcal C$ will be appropriate for equilibria of the threshold form \eqref{eq:thresholdform} in the Dynkin game. (The set $\mathcal C$ will be modified in Section \ref{sec:examples} below, where an example of a more complex equilibrium is studied).
Writing $U_i: \mathcal{C} \to \bar{\mathbb{R}}$ for the utility function of player $i$, the generalised Nash equilibrium problem is then given by:
\begin{definition}[GNEP, $n=2$]\label{def:gnep}
Find $s^{*} = (s_{1}^{*},s_{2}^{*}) \in \mathcal{C}$ which is a Nash equilibrium, that is:
\begin{equation}\label{eq:Auxiliary-GNEP}
\begin{cases}
U_1(s^{*}) = \sup\limits_{(s_1,s_2^*) \in \mathcal C} U_1(s_1,s_2^*), \\
U_2(s^{*}) = \sup\limits_{(s_1^*,s_2) \in \mathcal C} U_2(s_1^*,s_2).
\end{cases}
\end{equation}
\end{definition}
In the proofs below it will be convenient to write 
$\mathcal{S} := \mathcal{S}_{1} \times \mathcal{S}_{2}$. We will also make use of the following definition:

\begin{definition}
Let $s=(s_1,s_2, \ldots, s_n) \in \mathbb{R}^n$ and $w \in \mathbb{R}$. Then for each $i \in \{1,\ldots,n\}$ we will write $(w,s_{-i})$ for the vector $s$ modified by replacing its $i$th entry with $w$.
\end{definition}

\subsection{Optimal stopping}\label{sec:osg}
We also consider a {\em Dynkin game} in which two players observe the Brownian motion subprocess $X$ of Section \ref{sec:subpr}. Each player can stop the game and receive a reward (which may be positive or negative) depending on the process value and on who stopped the game first.
More precisely we consider only pure strategies: that is, each player $i \in \{1,2\}$ chooses a stopping time $\tau_i$ lying in $\mathcal{T}$ as their strategy.
%JM: This definition now moved up to its first usage:
%, the set of all $\mathbb{F}$-stopping times (which are permitted to take the value $+\infty$ so that $\mathcal T$ is the set of {\em Markov times} in the older literature).
Let $f_{i}$, $g_{i}$ and $h_{i}$ be real-valued {\em reward functions} on $E$ which respectively determine the reward to player $i$ from stopping first, second, or at the same time as the other player. For convenience we will refer to the $f_i$ as the {\em leader} reward functions and to the $g_i$ as the {\em follower} reward functions.
Given a pair of strategies $(\tau_{1},\tau_{2})$ and recalling the payoff defined in 
\eqref{eq:Game-Payoff-Functional}, 
%JM: I found it clearer to change (1.1)
%\RMA{using $X$ instead of $W$}, 
we denote the {\em expected payoff} to player $i$ by
\begin{equation}\label{eq:Expected-Game-Payoff-Functional}
M^{x}_{i}(\tau_{1},\tau_{2}) = \mathds{E}^{x}\left[\mathcal{J}_{i}(\tau_{1},\tau_{2})\right].
\end{equation} 
The problem of finding a Nash equilibrium for this Dynkin game is then:
\begin{definition}[DP]
	Find a pair $(\tau^{*}_{1},\tau^{*}_{2}) \in \mathcal{T} \times \mathcal{T}$ such that for every $x \in E$ we have:
	\begin{equation}\label{eq:NEP-Pure-Strategies}
	\begin{cases}
	M^{x}_{1}(\tau^{*}_{1},\tau^{*}_{2}) = \sup\limits_{\tau_{1} \in \mathcal{T}}M^{x}_{1}(\tau_{1},\tau^{*}_{2}) \\
	M^{x}_{2}(\tau^{*}_{1},\tau^{*}_{2}) = \sup\limits_{\tau_{2} \in \mathcal{T}}M^{x}_{2}(\tau^{*}_{1},\tau_{2}).
	\end{cases}
	\end{equation}
	If $\tau^{*}_{1} = D_{S_{1}}$ and $\tau^{*}_{2} = D_{S_{2}}$ with $S_{1}, S_{2} \in \mathcal{B}(E_{\Delta})$, then the Nash equilibrium $(D_{S_{1}},D_{S_{2}})$ is said to be {\em Markovian}.
\end{definition}

\subsection{Linking the games}
We now present the link between the games in the case $n=2$ and $E = (0,1)$, which is the setting used in the rest of the paper (with the exception of Section \ref{sec:examples}, where $n=3$). The idea is that threshold-type solutions to the DP can be characterised by the slopes $U_1(x,y)$ and $U_2(x,y)$ of certain secant lines. This gives nothing else than a deterministic game, which may be studied in the above generalised setting in order to discover additional novel equilibria. We will close this section by illustrating that this link between the DP and GNEP does not preserve the zero-sum property.
\subsubsection{Construction of utility functions for the GNEP}
For $(x,y) \in [0,1]^2$
 we define
% the utility functions $U_{i}$
%(Section \ref{sec:gnep}) using the reward functions $f_i$ and $g_i$
%(Section \ref{sec:osg}):
\begin{equation}\label{eq:Auxiliary-GNZSG-Utilities}
\begin{split}
U_{1}(x,y) & = \begin{cases}
\frac{f_{1}(x) - g_{1,[y,1]}(x)}{y - x},& x < y,\\
-\infty,& \text{otherwise},
\end{cases}\\
U_{2}(x,y) & = \begin{cases}
\frac{f_{2}(y) - g_{2,[0,x]}(y)}{y - x},& x < y,\\
-\infty,& \text{otherwise},
\end{cases}
\end{split}
%\mathcal{S},
\end{equation}
where for $A \in \mathcal B(E_{\Delta})$, the function $g_{i,A}$ is obtained by taking $\phi=g_i$ in Definition~\ref{def:red}. To ensure that these utility functions are continuous and bounded above on $\mathcal C$ we strengthen Assumption \ref{assumption:Preliminary-Assumption-1} to:
\smallskip

{\bf Assumption 1'}\label{Assumption:Infimum-Approach}
	If $b \le a$ then $g_{i} > f_{i}$ on $[b,a]$ for $i = 1,2$.

\smallskip	
%\noindent Also, where necessary we extend the domain of the functions $U_{i}$ from $\mathcal{C}$ to $\mathcal S$ by setting $U_i(s) = -\infty$ for all $s \notin \mathcal{C}$  \RMA{and formally set $\frac{z}{0} = -\infty$ for $z < 0$ in \eqref{eq:Auxiliary-GNZSG-Utilities}.}

%As we shall see below, this choice of utility functions corresponds to focusing on the geometry which characterises threshold type equilibria in the DP.
\noindent We now record comments on this choice of utility functions in the GNEP:

\indent (i) The rationale for the form \eqref{eq:Auxiliary-GNZSG-Utilities} of $U_1$ and $U_2$ is as follows. Lemma \ref{Lemma:Characterisation-1} below will confirm that in equilibrium, player 1's strategy is characterised by an optimal stopping problem with obstacle $f_{1} - g_{1,[r,1]}$, whose geometry determines the solution (in the sense of \cite{Dayanik2003}, for example). In particular we show in Theorem \ref{Theorem:Necessary-and-Sufficiency-Optimality} that, for threshold strategies, the function $U_1$ characterises the solution. Similar comments of course apply to player 2.

\indent (ii) The GNEP characterisation does not assume smoothness but, if the reward functions are differentiable, then the double smooth fit condition (that is, the differentiability of the players' equilibrium payoffs across the thresholds $\ell$ and $r$ respectively) follows as a corollary.

\indent (iii) Later, in Section \ref{sec:examples}, we show how additional functions $U_i$ may be added to characterise more complex equilibria than the threshold type, leading to GNEPs with more than two players.

\subsubsection{Remark on the zero-sum property}
It is interesting to note that the zero-sum property in the DP does not imply the same for the GNEP and vice versa. Suppose that the GNEP \eqref{eq:Auxiliary-GNZSG-Utilities} has zero sum: that is, 
\begin{equation}\label{eq:zsgnep}
\sum_{i=1}^{2}U_{i}(x,y) = 0,\quad\forall \, (x,y) \in \mathcal{S}.
\end{equation}
By Definition~\ref{def:red} the functions $g_{1,[y,1]}$ and $g_{2,[0,x]}$ are given by:
\begin{align}\label{eq:gexp1}
g_{1,[y,1]}(x) & = \begin{cases}
g_{1}(y)\cdot\frac{x}{y},& \forall x \in [0,y) \\
g_{1}(x),& \forall x \in [y,1],
\end{cases} \\ \label{eq:gexp2}
g_{2,[0,x]}(y) & = \begin{cases}
g_{2}(y),& \forall y \in [0,x] \\
g_{2}(x)\cdot\frac{1-y}{1-x},& \forall y \in (x,1],
\end{cases}
\end{align}
and we recall that $f_{1}(0) = g_{2}(0) = g_{1}(1) = f_{2}(1) = 0$. Then considering separately the case $x=0$, $y \in [b,1]$ in \eqref{eq:zsgnep} and the case $y=1$, $x \in [0,a]$, we conclude that $f_{1}(x) = f_{2}(y) = 0,\,\forall \, (x,y) \in \mathcal{S}$. Then in the DP, any nonzero choice of the reward functions $g_i$ satisfying Assumption~\ref{assumption:Preliminary-Assumption-1} results in a game with $f_i \neq -g_{-i}$ and hence is nonzero sum. 

\begin{comment}
Suppose $a < b$, 
holds and additionally,
	\begin{align*}
	\forall x \in [0,a] \colon \quad & \begin{cases}
	f_{1}(x) = 0, \\
	g_{2}(x) = x(x-1). 
	\end{cases} \\
	\forall y \in [b,1]\colon \quad & \begin{cases}
	f_{2}(y) = 0, \\
	g_{1}(y) = y(1-y). 
	\end{cases}
	\end{align*}
	Then $\sum_{i=1}^{2}U_{i}(x,y) = 0,\,\forall (x,y) \in \mathcal{S}$, but the DP is not zero-sum. 
\end{comment}

On the other hand, suppose that $a < b$ and consider the zero-sum DP with reward functions
	\begin{align*}
	f_{1}(x) & = \begin{cases}
	x(a-x),& x \in [0,a] \\
	(1-x)(a-x),&x \in (a,1],
	\end{cases} \\
	g_{1}(x) & = \begin{cases}
	x(b-x),& x \in [0,b) \\
	(1-x)(b-x),&x \in [b,1],
	\end{cases}\\
f_{2}& = -g_{1}, \quad g_{2} = -f_{1}, \quad h_{1} = -h_{2}.
	\end{align*}
	Then for $(x,y) \in \mathcal{S}$	the sum of the payoffs in the GNEP is
	\[
	\sum_{i=1}^{2}U_{i}(x,y) = x\left(\frac{a-x}{y-x}\right)\left(1+\frac{1-y}{1-x}\right) - \left(\frac{(1-y)(b-y)}{y-x}\right)\left(\frac{y+x}{y}\right),
	\]
which is strictly positive for $(x,y) \in \{0,a\} \times (b,1)$, and so the GNEP is not zero sum.

\section{Optimal stopping of a subprocess}\label{sec:fktsp}

In this section we 
%JM: The following doesn't seem to be used below, so has been commented out; also perhaps it's simply a consequence of the subprocess property?:
%Each $X$ is a time-homogeneous standard Markov process (see \cite[Definition~\MakeUppercase{\romannumeral 1}.9.2]{Blumenthal1968}) which we write as
%$$\mathscr{X} = \{\Omega,\mathcal{F},\mathbb{F}, \theta_{t}, X_{t},\mathds{P}^{x}\}.$$ 
 provide three equivalent expressions for expected payoffs in the Dynkin game, as optimal stopping problems for subprocesses. These will be used repeatedly to establish the existence and uniqueness results of Sections~\ref{Section:Existence-Of-Equilibria} and \ref{Section:Stability-Uniqueness}.

\subsection{Preliminaries}\label{sec:pre}

%%%%%%%%%%%%%%%
We begin by confirming that without loss of generality all reward functions may be set equal to zero on $\Delta$. Suppose instead that the reward is to be nonzero on $\Delta$. This could be accommodated by taking the following modified form for the expected payoffs:
\begin{align}
\begin{split}
	M^{x}_{E}(\tau,\sigma) = {} & \hat{\mathds{E}}^{x}\bigl[
	%e^{-\lambda (\tau \wedge \sigma)}
	\bigl\{f(W_{\tau})\mathds{1}_{\{\tau < \sigma\}} + g(W_{\sigma})\mathds{1}_{\{\tau > \sigma\}} + h(W_{\sigma})\mathds{1}_{\{\tau = \sigma\}}\bigr\}\mathds{1}_{\{(\tau \wedge \sigma) < D_{E^{c}}\}}\bigr] \\
	& + \hat{\mathds{E}}^{x}\bigl[
	%e^{-\lambda D_{E^{c}}}
	H(W_{D_{E^{c}}})\mathds{1}_{\{(\tau \wedge \sigma) \ge D_{E^{c}}\}}\bigr],
	\end{split}\label{eq:withH}
\end{align}
where $\tau,\sigma$ denote the players' stopping times and $H$ specifies the reward received at the boundaries of $E$. %allowing the players to choose stopping times for the Brownian motion $W$ and writing the subprocess $X$ as $(W_{t \wedge \rho_{E}})_{t \ge 0}$, where 
%$$\rho_{E} = \inf\{t \ge 0 \colon W_{t} \notin E\} = D_{E^{c}}$$
 %is the first exit time of the process $W$ from the set $E$. In this context, if 
%JM: Previously we set $E=(0,1)$ so the following is a bit onconsistent:
%This applies to situations where $E = \mathbb{R}$, $X$ starts at $x \in E = (\underline{x},\overline{x})$ for real numbers $\underline{x} < \overline{x}$, and is `absorbed' upon entering the set $\{\underline{x},\overline{x}\}$. 
%Thus the general form of the expected payoffs in \eqref{eq:Expected-Game-Payoff-Functional} would become
% and the reward functions $f$, $g$ and $h$ are zero on $\Delta$.
%, and $H$ is a bounded measurable reward function %on $E_\Delta$ 
This is for example the approach taken in \cite{Attard2015}, where it is assumed that $f(x) = g(x) = H(x),\, x \in E^{c}$. Note also that by construction only the values of $H$ on $E^c$ are relevant. Now taking $\phi=H$ and $A=E^{c}$ in Definition \ref{def:red}
% \colon E \to \mathbb{R}$
%\[
%H_{E^{c}}(x) = \mathds{E}^{x}\left[e^{-\lambda D_{E^{c}}}H(X_{D_{E^{c}}})\right],
%\]
and using the strong Markov property we can show that, %(see \cite[p.~346--347]{Bensoussan1982}, for example) that
\begin{align}
	M^{x}_{E}(\tau,\sigma) - H_{E^{c}}(x)  = {} & \hat{\mathds{E}}^{x}\bigl[
	%e^{-\lambda (\tau \wedge \sigma)}
	\bigl\{[f - H_{E^{c}}](W_{\tau})\mathds{1}_{\{\tau < \sigma\}} + [g - H_{E^{c}}](W_{\sigma})\mathds{1}_{\{\tau > \sigma\}} \nonumber \\
	& \qquad \qquad + [h - H_{E^{c}}](W_{\sigma})\mathds{1}_{\{\tau = \sigma\}}\bigr\}\mathds{1}_{\{(\tau \wedge \sigma) < D_{E^{c}}\}}\bigr] \nonumber \\
	= {} & \mathds{E}^{x}\bigl[
	%e^{-\lambda (\tau \wedge \sigma)}
	\bigl\{[f - H_{E^{c}}](X_{\tau})\mathds{1}_{\{\tau < \sigma\}} + [g - H_{E^{c}}](X_{\sigma})\mathds{1}_{\{\tau > \sigma\}} \nonumber \\
	& \qquad \qquad + [h - H_{E^{c}}](X_{\sigma})\mathds{1}_{\{\tau = \sigma\}}\bigr\}\mathds{1}_{\{(\tau \wedge \sigma) < D_{E^{c}}\}}\bigr], \label{eq:tsfH}
\end{align}
where the second equality comes from \eqref{eq:Vanishing-After-Death} above.
The right-hand side of \eqref{eq:tsfH} is equal to the expected payoff \eqref{eq:withH} when the reward functions $f$, $g$, $h$ and $H$ are taken to be $\tilde{f} = f - H_{E^{c}}$, $\tilde{g} = g - H_{E^{c}}$, $\tilde{h} = h - H_{E^{c}}$ and $\tilde H \equiv 0$ respectively.
It may therefore be assumed without loss of generality in the proofs below that the reward functions are zero on $\Delta$ (and indeed on $E^c$).
%JM: On reflection perhaps the following remark is a bit trivial?:
%Note that this choice is consistent with \cite{Attard2015}, since the latter setting gives $\tilde{f} = \tilde{g} = 0$ on $\Delta$.

We will be interested in optimally stopping the subprocess $X^{A^{c}}$. For this, define the set of stopping times $\mathcal{T}_{0,D_{A}} \coloneqq \{\tau \in \mathcal{T} \colon 0 \le \tau \le D_{A}\}$. The proof of the following useful result can be found in, for example, \cite{Attard2016} and \cite{Ekstrom2008}:

\begin{proposition}\label{Prop:Superaveraging-Property}
	For $A \in \mathcal{B}(E_\Delta)$ and functions $f$, $g$ and $h$ satisfying Assumption~\ref{assumption:Preliminary-Assumption-1}, the map 
	$$x \mapsto \check V(x)\coloneqq \sup_{\tau \in \mathcal{T}}\mathds{E}^{x}\left[
f(X_{\tau})\mathds{1}_{\{\tau < D_A\}} + g(X_{D_A})\mathds{1}_{\{D_A < \tau\}} + h(X_{D_A})\mathds{1}_{\{\tau = D_A\}}
\right],$$
	%defined in \eqref{eq:Single-Payoff-OSP} 
	is measurable and satisfies:
	\begin{equation}\label{eq:Alt-Rep-Super-Averaging}
	\forall  \rho \in \mathcal{T}_{0,D_{A}} \enskip \colon \enskip  \mathds{E}^{x}[\check V(X_{\rho})] \le \check V(x) \quad \forall x \in E.
	\end{equation}
	In other words, $x \mapsto \check V(x)$ is superharmonic on $A^{c}$.
\end{proposition}

\subsection{Single player problem}\label{Section:Single-Player-OSP}

Suppose that in the Dynkin game, the strategy of player $-i$ is specified by a set $A \in \mathcal{B}(E_\Delta)$ on which that player stops. The next lemma expresses the resulting optimisation problem for player $i$ in terms of optimal stopping problems of different kinds for the subprocess $X^{A^{c}}$.

\begin{lemma}\label{Lemma:Characterisation-1}
For $x \in E$ consider the problems
\begin{align}\label{eq:Single-Payoff-OSP}
V^{A}(x) & \coloneqq \sup_{\tau \in \mathcal{T}}M^{x}(\tau,D_{A}), \\
\bar V^{A}(x) & \coloneqq \sup_{\tau \in \mathcal{T}}\bar M^{x}(\tau,D_{A}), \\
\tilde V^A(x) & \coloneqq \sup_{\tau  \in \mathcal{T}}\tilde M^{x}(\tau,D_{A}), \label{eq:Single-Payoff-OSP-Alt-2-2}
\end{align}
where for $\tau \in \mathcal{T}$ we have
\begin{align}\label{eq:Single-Payoff-Functional}
M^{x}(\tau,D_A) &\coloneqq \mathds{E}^{x}\left[
%e^{-\lambda (\tau \wedge D_A) }\left\{
f(X_{\tau})\mathds{1}_{\{\tau < D_A\}} + g(X_{D_A})\mathds{1}_{\{D_A < \tau\}} + h(X_{D_A})\mathds{1}_{\{\tau = D_A\}}%\right\}
\right],\\
\label{eq:Single-Payoff-Alt-2}
	\bar{M}^{x}(\tau,D_A) &\coloneqq \mathds{E}^{x}\bigl[f(X_{\tau})\mathds{1}_{\{\tau < D_A\}} + g(X_{D_A})\mathds{1}_{\{\tau \ge D_A\}}\bigr],\\
	\tilde{M}^{x}(\tau,D_A) &\coloneqq \mathds{E}^{x}\bigl[	\bigl\{f - g_{A}\bigr\}(X_{\tau})\mathds{1}_{\{\tau < D_A\}}\bigr], \label{eq:deftildem}
\end{align}
and $f$, $g$ and $h$ are functions satisfying Assumption~\ref{assumption:Preliminary-Assumption-1}. Then, recalling Definition \ref{def:red}, we have
	\begin{align}
	V^{A}(x) &= \bar{V}^{A}(x)  
	%%\coloneqq \sup_{\tau \in \mathcal{T}}\mathds{E}^{x}\bigl[
	%e^{-\lambda (\tau \wedge D_{A})}
	%\bigl\{
	%%f(X_{\tau})\mathds{1}_{\{\tau < D_{A}\}} + g(X_{D_{A}})\mathds{1}_{\{\tau \ge D_{A}\}}
	%\bigr\}
	%%\bigr]\label{eq:Single-Payoff-OSP-Alt-2} \\
 = g_{A}(x) + \tilde V^A(x).	
 \label{eq:equivpayoffs}\end{align} 
\end{lemma}
\begin{proof}

	%It is clear 
	%from \eqref{eq:Single-Payoff-Functional} %, \eqref{eq:Single-Payoff-Alt-2} 
	%JM: f,g,h are now defined within the Lemma itself:
	%and Assumption~\ref{assumption:Preliminary-Assumption-1} 
	
Let $\tau \in \mathcal{T}$, $x \in E$ be arbitrary. We have $\bar{M}^{x}(\tau,D_{A}) \ge M^{x}(\tau,D_{A})$ and therefore $\bar{V}^{A}(x) \geq V^{A}(x) $. To show the reverse inequality,
	%JM: It's already sufficiently clear:
	%, $\bar{V}^{A}(x) \le V^{A}(x)$, 
first recall that $x \mapsto V^{A}(x)$ is measurable. 
%Let $\tau \in \mathcal{T}$ be arbitrary and notice that $\tau \wedge D_{A} \in \mathcal{T}_{0,D_{A}}$. %Using equation~\eqref{eq:Super-Averaging-2},
%	Proposition~\ref{Prop:Superaveraging-Property} with $\rho \equiv \tau \wedge D_{A}$, and 
%	the $\mathcal{F}_{\tau \wedge D_{A}}$-measurability of the payoff, $\mathds{P}$-a.s.,
%	\[
%	V^{A}(X_{\tau \wedge D_{A}}) \ge f(X_{\RMA{\tau}})\mathds{1}_{\{\tau < D_{A}\}} + g(X_{D_{A}})\mathds{1}_{\{D_{A} < \tau\}} + h(X_{D_{A}})\mathds{1}_{\{\tau = D_{A}\}},
%	\]
%	and multiplying both sides by $\mathds{1}_{\{\tau < D_{A}\}}$ gives
By assumption we have $V^A \geq f$ on $E$, so that $V^{A}(X_{\tau})\mathds{1}_{\{\tau < D_{A}\}} \ge f(X_{\tau})\mathds{1}_{\{\tau < D_{A}\}}$ a.s., while from the strong Markov property we have $V^{A}(X_{D_{A}}) = g(X_{D_{A}})$ a.s.. It follows from 
	%Using equations~\eqref{eq:Single-Payoff-OSP-Inequality-1} and \eqref{eq:Single-Payoff-OSP-Inequality-2} in 
	\eqref{eq:Single-Payoff-Alt-2} and superharmonicity that
	%applying Proposition~\ref{Prop:Superaveraging-Property} with $\rho \equiv \tau \wedge D_{A}$ that 
	%for every $x \in E$ and $\tau \in \mathcal{T}$,
	\[
	\bar{M}^{x}(\tau,D_{A}) \le \mathds{E}^{x}\bigl[V^{A}(X_{\tau \wedge D_{A}})\bigr] \le V^{A}(x),
	\]
	and taking the supremum over $\tau$ we have $\bar{V}^{A}(x) = V^{A}(x)$. Finally, recalling Definition \ref{def:red} we have
	%JM: Isn't this really just a straightforward subtraction of $g_A$, and use of indicator functions?:
	%and note that by using the strong Markov property and shift operator one can show that
	\begin{equation}
	\bar{M}(\tau,D_{A}) - g_{A}(x) = \mathds{E}^{x}\bigl[\bigl\{f - g_{A}\bigr\}(X_{\tau})\mathds{1}_{\{\tau < D_{A}\}}\bigr].
	\end{equation}
\end{proof}

\begin{remark}\label{Remark:Necessary-Condition-For-Stopping}
It follows from \eqref{eq:equivpayoffs} that
	\[
	V^A(x) = f(x) \iff \tilde V^A(x) = f(x) - g_{A}(x).
	\]
	That is, defining the {\em stopping region} to be the subset of $A^c$ on which the obstacle equals the value function, the optimal stopping problems $V^A(x)$ and $\tilde V^A(x)$ have identical stopping regions. An easy consequence is that if $x \in A^c$ lies in either stopping region then $f(x) \ge g_{A}(x)$, and that if $f \le g_{A}$ on $A^{c}$ then $\tau = D_A$ is optimal in \eqref{eq:Single-Payoff-OSP-Alt-2-2}.
\end{remark}

\section{Existence of equilibria}\label{Section:Existence-Of-Equilibria}

In this section we exploit the link between the games to show, firstly, that the existence of a solution to the GNEP with utility functions given by \eqref{eq:Auxiliary-GNZSG-Utilities} implies the existence of a threshold-type solution to the DP (Theorems \ref{Theorem:Existence-Generalised-Nash-Equilibrium-Stopping} and  \ref{Theorem:Necessary-and-Sufficiency-Optimality}). These results are then applied to show the existence of novel Nash equilibria in the DP. More precisely we will show that the following condition on the geometry of the reward functions is sufficient for the existence of an equilibrium:
\smallskip

{\bf Condition G1.}
\label{Assumption:Threshold-Geometry}
	There exist points $a \in (0,1) \text{ and } b \in (0,1) \text{ such that}$
	\begin{align*}
	(i) \quad & f_{1} \text{ is concave on } [0,a] \text{ and is convex on } [a,1] \\
	(ii) \quad & f_{2} \text{ is convex on } [0,b] \text{ and is concave on } [b,1]\\
	(iii) \quad & \text{If } b \le a, \text{ then } f_{i} < g_{i} \text{ on } [b,a] \text{ for } i = 1,2.
	\end{align*}

\smallskip
The case $a > b$ is novel when compared with the existing literature. It is interesting to note that in the case $a \le b$, which is analysed in \cite{Attard2015} and \cite{DeAngelis2015}, the generalised problem~\eqref{eq:Auxiliary-GNEP} reduces to a classical game (that is, where each player's strategy space does not depend on the other player's chosen strategy). This is because the dependency between the players' strategies (which is specified by the choice of $\mathcal{C}$) allows some additional control on the equilibria, which is required when $a>b$. The case when at least one of the functions $f_i$ is not differentiable is also novel.

\subsection{Preliminaries}

A solution to the GNEP is known to exist under the following condition (see for example \cite{Arrow1954} and \cite{Facchinei2007}):
\begin{description}
\item[\textbf{Condition U.}]\mbox{}
\begin{enumerate}[(i)] 
		\item For each fixed $s_2 \in \mathcal{S}_{2}$, the mapping $s_{1} \mapsto U_{1}(s_1,s_2)$ is quasi-concave on $K_1(s_2)$. For each fixed $s_1 \in \mathcal{S}_1$, the mapping $s_{2} \mapsto U_{2}(s_1,s_2)$ is quasi-concave on $K_2(s_1)$. 
		\item The utility functions $s \mapsto U_{i}(s)$ for $i = 1,2$ are continuous in $s = (s_{1},s_{2})$.
	\end{enumerate}
\end{description}
For convenience we record the necessary argument here:
\begin{lemma}\label{Lemma:Existence-Solution-GNEP}
Suppose Assumption~1' and Condition U hold. Then there exists a solution $\bigl(s^{*}_{1},s^{*}_{2}\bigr) \in \mathcal{C}$ to the GNEP \eqref{eq:Auxiliary-GNEP} satisfying $s_{1}^{*} < s_{2}^{*}$.
\end{lemma}
\begin{proof}
	For $i = 1,2$ the correspondence $K_{i}$ is compact and convex valued. Furthermore, using the notion of continuity for set-valued maps in \cite{Rockafellar1998}, we can confirm that $K_{1}$ and $K_{2}$ are continuous. Under the present hypotheses, $U_{i}$ is continuous on $\mathcal{S}$ and has the quasi-concavity property specified in Condition U.
	Therefore by Lemma 2.5 in \cite{Arrow1954}, there exists a solution $s^{*}$ to \eqref{eq:Auxiliary-GNEP}. From the construction \eqref{eq:Auxiliary-GNZSG-Utilities}, this solution must satisfy $s_{1}^{*} < s_{2}^{*}$.% Otherwise, we would have $s_{1}^{*} = s_{2}^{*}$ due the definition of $\mathcal{C}$, and $U_{1}(s^{*}) = -\infty$ due to \eqref{eq:Auxiliary-GNZSG-Utilities}, which contradicts \eqref{eq:Auxiliary-GNEP} since $U_{1}(s^{*}) < U_{1}(x,s^{*}_{2})$ for $x \in [0,s^{*}_{2})$.}
	%Furthermore, under Assumption~\ref{Assumption:Infimum-Approach} the point $s^{*}$ necessarily satisfies $s_{1}^{*} < s_{2}^{*}$ (otherwise it would fail to be a maximum).
\end{proof}

\begin{remark}
For possible extensions of Lemma \ref{Lemma:Existence-Solution-GNEP} see also \cite{Facchinei2007,He2016} and references therein.
\end{remark}
 
Before presenting the main result of this section we need the following fact:

\begin{lemma}\label{Lemma:Quasi-Concavity}
	Suppose $\mathcal{D} \subseteq \mathbb{R}$ is convex, $f \colon \mathcal{D} \to \bar{\mathbb{R}}$ is (strictly) concave, and $\varphi \colon \mathcal{D} \to (0,\infty)$ is linear. Then the function $\frac{f}{\varphi} \colon \mathcal{D} \to \bar{\mathbb{R}}$ is (strictly) quasi-concave.
\end{lemma}
\begin{proof}
	In the case of concavity, for each $\alpha \in \mathbb{R}$ define a function $F_{\alpha} \colon \mathcal{D} \to \bar{\mathbb{R}}$ by $F_{\alpha}(x) = f(x) - \alpha \varphi(x)$. This function is concave on $\mathcal{D}$, and therefore quasi-concave, which means the superlevel set $\left\{x \in \mathcal{D} \colon F_{\alpha}(x) \ge 0 \right\}$ is convex for every $\alpha \in \mathbb{R}$. The function $\frac{f}{\varphi}$ is quasi-concave on $\mathcal{D}$ since for every $\alpha \in \mathbb{R}$,
	\[
	\left\{x \in \mathcal{D} \colon \left(\tfrac{f}{\varphi}\right)(x) \ge \alpha \right\} = \left\{x \in \mathcal{D} \colon f(x) \ge \alpha \varphi(x) \right\} = \left\{x \in \mathcal{D} \colon F_{\alpha}(x) \ge 0 \right\}.
	\]
	The proof for strictly concave $f$ follows in the same way.
\end{proof} 

\subsection{Existence results}
The GNEP is used in this section to establish the existence of equilibria in the DP. This both allows $a>b$ in Condition G1 and avoids the need for smoothness assumptions.

\begin{theorem}\label{Theorem:Existence-Generalised-Nash-Equilibrium-Stopping}
	Under Condition G1, there exists a pair $(\ell_{*},r_{*}) \in [0,a] \times [b,1]$ such that $(D_{[0,\ell_{*}]},D_{[r_{*},1]})$ is a solution to the DP.
\end{theorem}

\begin{proof}
We begin by noting that for each $r \in [0,1]$ and $\ell \in [0,1]$,
	\begin{align}
	\sup_{x \in [0,r)} U_{1}(x,r) &\leq \sup_{x \in [0,a]} U_{1}(x,r), \label{eq:ext1}\\
		\sup_{x \in (\ell,1]} U_{2}(\ell,x) &\leq \sup_{x \in [b,1]} U_{2}(\ell,x). \label{eq:ext2}
	\end{align}
For $r \in (a,1]$, eq. \eqref{eq:ext1} follows from the 
convexity of $f_{1} - g_{1,[r,1]}$ on $[a,r]$ and the fact that $f_{1}(r) \le g_{1}(r) = g_{1,[r,1]}(r)$:
	\begin{align*}
	\frac{f_{1}(x) - g_{1,[r,1]}(x)}{r - x} & \le \frac{f_{1}(a) - g_{1,[r,1]}(a)}{r - a} + \left(\frac{f_{1}(r) - g_{1,[r,1]}(r)}{r - a}\right)\left(\frac{x-a}{r-x}\right)  \\
	& \le \frac{f_{1}(a) - g_{1,[r,1]}(a)}{r - a}, \quad \forall x \in (a,r). \\
	\end{align*}
Similar reasoning establishes \eqref{eq:ext2}.

Using Condition G1 and Lemma~\ref{Lemma:Quasi-Concavity}, we can verify the hypotheses of 	
Lemma~\ref{Lemma:Existence-Solution-GNEP} and assert the existence of a pair $\bigl(\ell,r\bigr) \in [0,a] \times [b,1]$ with $\ell < r$ such that
	\begin{equation}\label{eq:Nash-Equilbrium-Deterministic-Game}
	\begin{cases}
	U_{1}(x,r) \le U_{1}(\ell,r),
	 \quad \forall x \in [0,r \wedge a], \\
	 	U_{2}(\ell,y) \le U_{2}(\ell,r),
	\quad \forall y \in [\ell \vee b,1].
	\end{cases}
	\end{equation}
The pair $(\ell,r)$ that satisfies \eqref{eq:Nash-Equilbrium-Deterministic-Game} therefore also satisfies 
		\begin{align}\label{eq:ex1}
	U_{1}(x,r) &\le U_{1}(\ell,r),
	 \quad \forall x \in [0,r), \\
	 	U_{2}(\ell,y) &\le U_{2}(\ell,r),
	\quad \forall y \in (\ell,1], \label{eq:ex2}
	\end{align}
	and the result follows from the next theorem.
\end{proof}
	
\begin{theorem}\label{Theorem:Necessary-and-Sufficiency-Optimality}
	For every $r \in [b,1]$, a point $\ell_{r} \in [0,a]$ with $\ell_{r} < r$ satisfies \eqref{eq:ext1} if and only if
	\begin{equation}\label{eq:Threshold-Optimisation-1}
	V_{1}^{[r,1]}(x) \coloneqq \sup\limits_{\tau_{1} \in \mathcal{T}}M^{x}_{1}(\tau_{1},D_{[r,1]}) = M^{x}_{1}(D_{[0,\ell_{r}]},D_{[r,1]}),\quad \forall x \in [0,1].
	\end{equation}
	Similarly, for every $\ell \in [0,a]$, a point $r_{\ell} \in [b,1]$ with $\ell < r_{\ell}$ satisfies \eqref{eq:ext2} if and only if
	\begin{equation}\label{eq:Threshold-Optimisation-2}
	V_{2}^{[0,\ell]}(x) \coloneqq \sup\limits_{\tau_{2} \in \mathcal{T}}M^{x}_{2}(D_{[0,\ell]},\tau_{2}) = M^{x}_{2}(D_{[0,\ell]},D_{[r_{\ell},1]}),\quad \forall x \in [0,1].
	\end{equation}
\end{theorem}
\begin{proof}
	We only show \eqref{eq:ex1}$\iff$\eqref{eq:Threshold-Optimisation-1}, since \eqref{eq:ex2}$\iff$\eqref{eq:Threshold-Optimisation-2} follows by similar arguments.	Let $r \in [b,1]$ and $\ell_{r} \in [0,a]$ with $\ell_{r} < r$ be given. We will make repeated use of the function 
	\begin{equation}\label{eq:Threshold-Auxiliary-Optimisation-Proof-3_0}
u_{r}(x) \coloneqq	M^{x}_{1}(D_{[0,\ell_{r}]},D_{[r,1]}) - g_{1,[r,1]}(x) = 
	\begin{cases}
f_{1}(x) - g_{1,[r,1]}(x),& x \in [0,\ell_{r}), \\
\left(f_{1}(\ell_{r}) - g_{1,[r,1]}(\ell_{r})\right)\frac{r-x}{r-\ell_{r}},& x \in [\ell_{r},r), \\
0,& x \in [r,1],
\end{cases}
	\end{equation}
where the middle line is a straightforward consequence of the identities in
Appendix~\ref{Section:Payoff-Derivation} and the fact that, for $x \in [0,r]$, we have
	\begin{align}
	g_{1}(r)\frac{x-\ell_{r}}{r-\ell_{r}} - g_{1,[r,1]}(x) & = g_{1}(r)\left(\frac{x-\ell_{r}}{r-\ell_{r}} - \frac{x}{r}\right) \nonumber\\
	& = g_{1}(r)\left(\frac{r(x-\ell_{r}) - x(r-\ell_{r})}{r(r-\ell_{r})}\right) \nonumber\\
	& = -g_{1}(r)\left(\frac{\ell_{r}}{r}\right)\left(\frac{r-x}{r-\ell_{r}}\right) = -g_{1,[r,1]}(\ell_{r})\frac{r-x}{r-\ell_{r}}.\label{eq:Threshold-Auxiliary-Optimisation-Proof-2}
	\end{align}

		\vskip0.5em 

		Sufficiency ($\impliedby$).
				
Suppose that \eqref{eq:Threshold-Optimisation-1} is satisfied. Substituting this in \eqref{eq:Threshold-Auxiliary-Optimisation-Proof-2},
\begin{comment}
		\vskip0.5em
	%Let $r \in [b,1]$ be given and suppose there exists a point $\ell_{r} \in [0,a]$, $\ell_{r} < r$, such that:
	\begin{align}
 	V_{1}^{[r,1]}(x)  - g_{1,[r,1]}(x) & = M^{x}_{1}(D_{[0,\ell_{r}]},D_{[r,1]}) - g_{1,[r,1]}(x) \nonumber \\
	& = \begin{cases}
	f_{1}(x) - g_{1,[r,1]}(x),& \forall x \le \ell_{r} \\
	f_{1}(\ell_{r})\frac{r-x}{r-\ell_{r}} + g_{1}(r)\frac{x-\ell_{r}}{r-\ell_{r}} - g_{1,[r,1]}(x),&\forall \ell_{r} < x \le r
	\end{cases}\label{eq:Threshold-Auxiliary-Optimisation-Proof-1}
	\end{align}
	See Appendix~\ref{Section:Payoff-Derivation} for a derivation of the second equality in \eqref{eq:Threshold-Auxiliary-Optimisation-Proof-1}. The function $g_{1,[r,1]}$ (cf. Definition \ref{def:red}) is explicitly given by,
	\begin{align}
g_{1,[r,1]}(x) & = \begin{cases}
g_{1}(r)\cdot\frac{x}{r},& \forall x \in [0,r) \\
g_{1}(x),& \forall x \in [r,1],
\end{cases} \label{eq:Player-1-Reduced-Function-Threshold-Strategies}
\end{align} 
yielding
	\begin{align}
	g_{1}(r)\frac{x-\ell_{r}}{r-\ell_{r}} - g_{1,[r,1]}(x) & = g_{1}(r)\left(\frac{x-\ell_{r}}{r-\ell_{r}} - \frac{x}{r}\right) \nonumber\\
	& = g_{1}(r)\left(\frac{r(x-\ell_{r}) - x(r-\ell_{r})}{r(r-\ell_{r})}\right) \nonumber\\
	& = -g_{1}(r)\left(\frac{\ell_{r}}{r}\right)\left(\frac{r-x}{r-\ell_{r}}\right) = -g_{1,[r,1]}(\ell_{r})\frac{r-x}{r-\ell_{r}}.\label{eq:Threshold-Auxiliary-Optimisation-Proof-2}
	\end{align}
	Using \eqref{eq:Threshold-Auxiliary-Optimisation-Proof-2} in \eqref{eq:Threshold-Auxiliary-Optimisation-Proof-1} gives,
\end{comment}
	dividing both sides of \eqref{eq:Threshold-Auxiliary-Optimisation-Proof-3_0} by $r-x$ (when $x<r$), 
	%for $0 \le x < r$ 
	and using the definition \eqref{eq:Auxiliary-GNZSG-Utilities} of $U_{1}$, we obtain
%		\begin{equation}\label{eq:Threshold-Auxiliary-Optimisation-Proof-3}
%	V_{1}^{[r,1]}(x) - g_{1,[r,1]}(x) = \begin{cases}
%	f_{1}(x) - g_{1,[r,1]}(x),& \forall x \le \ell_{r} \\
%	\left(f_{1}(\ell_{r}) - g_{1,[r,1]}(\ell_{r})\right)\frac{r-x}{r-\ell_{r}},&\forall \ell_{r} < x \le r.
%	\end{cases}
%	\end{equation}
	\begin{equation}\label{eq:Threshold-Auxiliary-Optimisation-Proof-4}
	\frac{V_{1}^{[r,1]}(x) - g_{1,[r,1]}(x)}{r - x} = \begin{cases}
	U_{1}(x,r),& \forall x \le \ell_{r} \\
	U_{1}(\ell_{r},r),&\forall \ell_{r} < x < r.
	\end{cases}
	\end{equation}
	It is easy to see that $V_{1}^{[r,1]}(r) = g_{1}(r) = g_{1,[r,1]}(r)$ and $V_{1}^{[r,1]}(x) \ge f_{1}(x)$ for all $x \in [0,r]$. Therefore when $x \in (\ell_r,r)$ we have
	\begin{equation*}
	  U_{1}(\ell_r,r) \ge U_{1}(x,r).
	%, \quad \forall x \in (\ell_{r},r).
	\end{equation*}
To treat the case $x \in [0,\ell_{r}]$, note from Lemma \ref{Lemma:Characterisation-1} 
that 
%and Proposition \ref{Prop:Superaveraging-Property} in the Appendix 
$x \mapsto V_{1}^{[r,1]}(x) - g_{1,[r,1]}(x)$ is the value function of an optimal stopping problem for a subprocess as in, for example, \cite{Dayanik2003} and, as such, is non-negative and superharmonic in $(0,r)$. For $0 \le x < y \le 1$ define $\tau_{x,y} = D_{\{x\}} \wedge D_{\{y\}}$. Using superharmonicity and the fact that $X$ is a positively recurrent diffusion, for every $0 \le x \le \ell_{r}$ we have,
	\begin{align}
	V_{1}^{[r,1]}(\ell_{r}) - g_{1,[r,1]}(\ell_{r}) \ge {} & \mathds{E}^{\ell_{r}}\bigl[V_{1}^{[r,1]}(X_{\tau_{x,r}}) - g_{1,[r,1]}(X_{\tau_{x,r}})\bigr] \nonumber \\
	= {} & \left(V_{1}^{[r,1]}(x) - g_{1,[r,1]}(x)\right)\mathds{E}^{\ell_{r}}\bigl[\mathds{1}_{\{D_{\{x\}} < D_{\{r\}}\}}\bigr] \nonumber \\
%	& + \left(V_{1}^{[r,1]}(r) - g_{1,[r,1]}(r)\right)\mathds{E}^{\ell_{r}}\bigl[\mathds{1}_{\{D_{\{r\}} < D_{\{x\}}\}}\bigr] \nonumber \\
	= {} & \left(V_{1}^{[r,1]}(x) - g_{1,[r,1]}(x)\right)\frac{r-\ell_{r}}{r-x}.\label{eq:Threshold-Auxiliary-Optimisation-Proof-5}
	\end{align}
	Since for all $0 \le x \le \ell_{r}$ we have $V_{1}^{[r,1]}(x) = f_{1}(x)$, \eqref{eq:Threshold-Auxiliary-Optimisation-Proof-5} gives
	\[
	U_{1}(x,r) \le U_{1}(\ell_r,r), \quad \forall x \in [0,\ell_{r}],
	\]
	establishing \eqref{eq:ex1} with $\ell=\ell_r$.
\vskip0.5em 

Necessity ($\implies$).
\vskip0.5em
%Let $r \in [b,1]$ be given and 
Suppose that the pair $(\ell_r, r)$ satisfies \eqref{eq:ex1} with $\ell=\ell_r$.
% for some $\ell_{r} \in [0,a]$, $\ell_r < r$. 
% Then define $u_{r}$ on $[0,1]$ by
%\begin{align}\label{eq:Candidate-Solution-One-Point-Case}
%u_{r}(x) & = M^{x}_{1}(D_{[0,\ell_{r}]},D_{[r,1]}) - g_{1,[r,1]}(x). \nonumber
% \\ 
%& = \begin{cases}
%f_{1}(x) - g_{1,[r,1]}(x),& x \in [0,\ell_{r}), \\
%\left(f_{1}(\ell_{r}) - g_{1,[r,1]}(\ell_{r})\right)\frac{r-x}{r-\ell_{r}},& x \in [\ell_{r},r), \\
%0,& x \in [r,1].
%\end{cases}
%\end{align}
We will establish  \eqref{eq:Threshold-Optimisation-1} by showing that 
\begin{equation}\label{eq:showur}
	u_r(x)=V_{1}^{[r,1]}(x) - g_{1,[r,1]}(x),\quad \forall x \in [0,1].
\end{equation}
By construction \eqref{eq:showur} holds for $x \in [r,1]$, and so we restrict attention to the domain $[0,r]$.
By Lemma~\ref{Lemma:Characterisation-1} it is sufficient to show that $u_r$ is the value function of the optimal stopping problem on $[0,r]$ with the obstacle $\vartheta \coloneqq f_{1} - g_{1,[r,1]}$. Therefore using Proposition 3.2 in \cite{Dayanik2003}, it is enough to show that $u_{r}$ is the smallest non-negative concave majorant of $\vartheta$ on $[0,r]$.
  The majorant property on $[\ell_r,r)$ follows from \eqref{eq:ex1}, which gives
\begin{equation}\label{eq:majorantur}
f_{1}(x) - g_{1,[r,1]}(x) \le \left(f_{1}(\ell_{r}) - g_{1,[r,1]}(\ell_{r})\right)\left(\frac{r-x}{r-\ell_{r}}\right),\enskip \forall x \in [0,r],
\end{equation}
and the majorant property at $x=r$ follows from recalling that $f_{1}(r) \le g_{1}(r)$.
%By this inequality and the definition of $u_{r}$ in \eqref{eq:Candidate-Solution-One-Point-Case} we see that $u_{r} \ge f_{1} - g_{1,[r,1]}$ on $[0,r]$. 
For nonnegativity we first recall that the reward functions are null at the boundaries, so taking $x=0$ in \eqref{eq:majorantur} gives $0 \leq f_{1}(\ell_{r}) - g_{1,[r,1]}(\ell_{r}) = u_r(\ell_r)$. Combining this with the fact that $u_r$ equals the obstacle on $[0,\ell_r]$, and hence is concave there, establishes nonnegativity.
%Since $f_{1} - g_{1,[r,1]}$ is concave on $[0,a]$ we have $f_{1}(x) - g_{1,[r,1]}(x) \ge \frac{x}{\ell_{r}}[f_{1}(\ell_{r}) - g_{1,[r,1]}(\ell_{r})] \ge 0$ for all $x \in [0,\ell_{r}]$, and consequently $u_{r}$ is non-negative. 
For concavity we note that $u_r$ is a straight line on $[\ell_r,r]$, so it remains only to consider any $x_{1} \in [0,\ell_{r})$ and $x_{2} \in (\ell_{r},r]$. Then we have
\begin{align*}
\frac{x_{2} - \ell_{r}}{x_{2}-x_{1}}u_{r}(x_{1}) + \frac{\ell_{r}-x_{1}}{x_{2}-x_{1}}u_{r}(x_{2}) = {} & \frac{x_{2} - \ell_{r}}{x_{2}-x_{1}}[f_{1}(x_{1}) - g_{1,[r,1]}(x_{1})] \nonumber \\
& + \frac{\ell_{r}-x_{1}}{x_{2}-x_{1}}\left(f_{1}(\ell_{r}) - g_{1,[r,1]}(\ell_{r})\right)\left(\frac{r-x_{2}}{r-\ell_{r}}\right) \nonumber \\
\le {} & \frac{x_{2} - \ell_{r}}{x_{2}-x_{1}}\left(f_{1}(\ell_{r}) - g_{1,[r,1]}(\ell_{r})\right)\left(\frac{r-x_{1}}{r-\ell_{r}}\right) \nonumber \\
& + \frac{\ell_{r}-x_{1}}{x_{2}-x_{1}}\left(f_{1}(\ell_{r}) - g_{1,[r,1]}(\ell_{r})\right)\left(\frac{r-x_{2}}{r-\ell_{r}}\right) \nonumber \\
= {} & f_{1}(\ell_{r}) - g_{1,[r,1]}(\ell_{r}) = u_{r}(\ell_{r}),
\end{align*}
where the inequality follows from \eqref{eq:ex1}.
Finally, since $u_r$ equals the obstacle on $[0,\ell_r]$ and is a straight line on $[\ell_r,r]$, it is smaller than any other nonnegative concave majorant on $[0,r]$.
\end{proof}
\begin{remark}\label{Remark:Waiting-Incentive}
%JM: Not sure about the value of the first part of this remark:
%The ordering $f_{i} \le h_{i} \le g_{i}$ in Assumption~\ref{assumption:Preliminary-Assumption-1} is used mainly in the proof of Lemma~\ref{Lemma:Characterisation-1}. If we assume $h_{i} = g_{i}$, then Lemma~\ref{Lemma:Characterisation-1} holds trivially, and
Theorem~\ref{Theorem:Existence-Generalised-Nash-Equilibrium-Stopping} remains valid under Condition G1 and the following condition which is weaker than Assumption~\ref{assumption:Preliminary-Assumption-1}: $f_{i} \le g_{i}$ on $\mathcal{S}_{-i}$.
\end{remark}

\section{Stability and uniqueness results}\label{Section:Stability-Uniqueness}
In this section we exploit the above connection to obtain additional novel results for Nash equilibria in the DP. 
We define a concept of stability and provide a sufficient condition under which it holds locally (Corollary \ref{Theorem:Local-Stability-DP}), showing in Theorem \ref{Theorem:Local-Stability-In-Zero-Sum-Game} that this condition always holds in the particular case of zero-sum Dynkin games. By establishing global stability, Theorem \ref{Theorem:Global-Stability-Theorem} provides sufficient conditions for uniqueness of the threshold-type equilibrium of Theorem \ref{Theorem:Existence-Generalised-Nash-Equilibrium-Stopping} among the Markovian strategies. Finally, Theorem \ref{Theorem:Rosen-Uniqueness} transfers another uniqueness result for the GNEP to the DP.

%\bl{
%In contrast, the use of stability allows weaker conditions to be placed on the reward functions. Indeed, we show in Section \ref{sec:beyond} that more complex equilibria in the DP can be linked to GNEPs having more than two players. In principle one may then appeal to stability and uniqueness results for such many-player GNEPs to obtain results on the uniqueness and stability (properly defined) of such complex equilibria in the DP. See also Remark \ref{rem:xqc} for a possible extension to Theorem \ref{sec:glob2}.}

%JM: Not sure if the unfamiliar reader will get the point about g_i being a supermartingale, since we haven't previously explained this:
%, where the reward functions $g_{i}$, as stochastic processes when composed with $X$, were not required to be supermartingales. 

\subsection{Policy iteration}
We will apply the {\em Gauss-Seidel policy iteration} or {\em t\^{a}tonnement process} \cite{Fudenberg1991,Basar1998} to the GNEP. This iteration scheme has previously been used for Dynkin games in \cite{Cattiaux1990} and \cite{KARATZASSudderth2006} and, outside the Markovian framework, in \cite{Hamadene2010}. 
Throughout Section \ref{Section:Stability-Uniqueness}, for convenience we will assume the following Condition G1', rather than G1:

\smallskip
{\bf Condition G1'}.
\label{Assumption:Smooth-Fit}
	Condition G1 holds, with: 
	\begin{enumerate}
\item[1)] $a < b$,
\item[2)] strict convexity and strict concavity, 
\item[3)] $f_{i}, g_{i} \in C^{2}[0,1]$, and
\item[4)] For all $(x,y) \in [0,a] \times [b,1]$ there exists $(\hat{x},\hat{y}) \in (0,a]\times [b,1)$ with $f_{1}(\hat{x}) > g_{1}(y)\cdot\frac{\hat{x}}{y}$ and $f_{2}(\hat{y}) > g_{2}(x)\cdot\frac{1-\hat{y}}{1-x}$.
	\end{enumerate}
\smallskip

We emphasise that these assumptions are for ease of exposition. Parts 1) and 3) imply that the GNEP utility functions are finite and smooth on $\mathcal{S}$, which is convenient for the transfer of results from generalised games. Part 2) says that $f_{1}$ is strictly concave on $[0,a]$ and strictly convex on $[a,1]$, and $f_{2}$ is strictly convex on $[0,b]$ and strictly concave on $[b,1]$. This ensures that iteration (i) below is well defined. %JM: I guess the reason just given already justifies part 2) well enough and we don't need more reasons:
%, and is used in Corollary~\ref{Theorem:Local-Stability-DP} to connect \todo{JM: It is still used?} the stability of the GNEP and DP.
Part 4) removes the need to consider the points 0 and 1 as candidate thresholds, which is convenient since the principle of smooth fit (used below) may break down there. Recalling the equality \eqref{eq:equivpayoffs}, this is straightforward to see from \eqref{eq:Single-Payoff-OSP-Alt-2-2}, \eqref{eq:deftildem} and \eqref{eq:gexp1}--\eqref{eq:gexp2}. Part 4) similarly ensures that threshold-type equilibria have their thresholds in $(0,1)$ and not at either boundary $0$ or $1$.

Taking $\ell^{(1)} \in [0,a]$, we consider the following two iteration schemes:
\begin{enumerate}
\item[(i)] {\bf In the GNEP}: taking $r^{(1)} = \argmax_{y \in[b,1]}U_{2}(\ell^{(1)},y)$, for $n \geq 2$ define
\begin{equation}\label{eq:GNEPit}
\begin{split}
%\ell^{(1)} & 
%\in [0,a], 
%\ell^{(2)} & = \argmax_{x \in[0,a]}U_{1}(x,r^{(1)}), \quad r^{(2)} = \argmax_{y \in[b,1]}U_{2}(\ell^{(2)},y),\\
%\vdots & \\
\ell^{(n)} & = \argmax_{x \in[0,a]}U_{1}(x,r^{(n-1)}), \quad  r^{(n)} = \argmax_{y \in[b,1]}U_{2}(\ell^{(n)},y). \\
%\vdots &
\end{split}
\end{equation} 
\item[(ii)]  {\bf In the DP}:  taking $A_{1} = [0,\ell^{(1)}]$, for $n \geq 1$ define
	\begin{equation}\label{eq:Iterative-VF-Stopping}
	\begin{split} 
	(i)&\quad V_{2n}(x) = \sup\limits_{\tau}\mathds{E}^{x}\bigl[
	%e^{-\lambda( \tau \wedge D_{A_{2n-1}})}
	%\bigl\{
	f_{2}(X_{\tau})\mathds{1}_{\{\tau < D_{A_{2n-1}}\}} + g_{2}(X_{D_{A_{2n-1}}})\mathds{1}_{\{\tau \ge D_{A_{2n-1}}\}}
	%\bigr\}
	\bigr], \\
	(ii)&\quad A_{2n} = \{x \in [0,1] \setminus A_{2n-1} \colon  V_{2n}(x) = f_{2}(x)\}, \\
	(iii)&\quad V_{2n+1}(x) = \sup\limits_{\tau}\mathds{E}^{x}\bigl[
	%e^{-\lambda( \tau \wedge D_{A_{2n}})}\bigl\{
	f_{1}(X_{\tau})\mathds{1}_{\{\tau < D_{A_{2n}}\}} + g_{1}(X_{D_{A_{2n}}})\mathds{1}_{\{\tau \ge D_{A_{2n}}\}}
	%\bigr\}
	\bigr], \\
	(iv)&\quad A_{2n+1} = \{x \in [0,1] \setminus A_{2n} \colon  V_{2n+1}(x) = f_{1}(x)\}.
	\end{split}
	\end{equation}
\end{enumerate}

%\begin{definition}{(See \cite[p.~172]{Basar1998}.)}\label{def:gs}
We will call a solution $s^{*}=(\ell^*,r^*)$ to the GNEP~\eqref{eq:Auxiliary-GNEP} {\em globally stable} if for any $\ell^{(1)} \in [0,a]$ the iteration \eqref{eq:GNEPit} satisfies $\ell^{(n)} \to \ell^*$ and $r^{(n)} \to r^*$, and
%JM: It seems unnecessary to present this level of generality, since we then immediately specialise. I've tried stating the iteration directly here:
%with respect to an adjustment scheme $\mathbb{S}$ if, starting from any $s^{(0)} \in \mathcal{S}$,
%\begin{align*}
%s^{*} & = \lim_{n \to \infty}s^{(n)} \\
%s^{(n+1)}_{i} & = \argmax_{\substack{s_{i} \in \mathcal{S}_{i} \\ \bigl(s^{(\mathbb{S}_{n})}_{-i},s_{i}\bigr) \in \mathcal{C}}}U_{i}\bigl(s^{(\mathbb{S}_{n})}_{-i},s_{i}\bigr),\;i \in \{1,2\},
%\end{align*}	
%where the superscript $\mathbb{S}_{n}$ emphasises dependence of the choice for $s^{(\mathbb{S}_{n})}_{-i}$ on the scheme $\mathbb{S}$ (for example, $s^{(\mathbb{S}_{n})}_{-i} = s^{(n)}_{-i}$). 
{\em locally stable} if this convergence holds only for $\ell^{(1)}$ in a neighbourhood of $\ell^*$. 
Similarly we call a threshold-type solution $s'=(D_{[0,\ell']},D_{[r',1]})$ to the DP~\eqref{eq:NEP-Pure-Strategies} globally stable if for any $\ell^{(1)} \in [0,a]$ the iteration \eqref{eq:Iterative-VF-Stopping} satisfies
 		\begin{equation*}
 		\begin{split}
 		\liminf_{n \to \infty}A_{2n-1} & = \limsup_{n \to \infty}A_{2n-1} = [0,\ell'],\\
 		\liminf_{n \to \infty}A_{2n} & = \limsup_{n \to \infty}A_{2n} = [r',1],
		\end{split}
 		\end{equation*}
and locally stable if convergence holds only for $\ell^{(1)}$ in a neighbourhood of $\ell'$.

\subsection{Local stability}

We will appeal to the following local stability result for the GNEP:
\begin{proposition}[Theorem~1.2.3, \cite{Krasnoselskii1972}]\label{Proposition:Contraction-In-Neighbourhood-Of-Fixed-Point}
	Suppose that Condition G1' holds and that $(\ell_{*},r_{*}) \in (0,a) \times (b,1)$ is a solution to the GNEP. For $w \in \mathcal{S}_{1}$ set 
\begin{equation}\label{eq:Best-Responses}
	\begin{split}
	\bar{y} & = \bar{y}(w) = \argmax_{y \in \mathcal{S}_{2}}U_{2}(w,y), \\
	\bar{x} & = \bar{x}(w) = \argmax_{x \in \mathcal{S}_{1}}U_{1}(x,\bar{y}(w)),
	\end{split}
\end{equation}
and
\[
T(w,\bar{x},\bar{y}) \coloneqq \frac{\partial_{xy}U_{1}(\bar{x},\bar{y})}{\partial_{xx}U_{1}(\bar{x},\bar{y})} \frac{\partial_{xy}U_{2}(w,\bar{y})}{\partial_{yy}U_{2}(w,\bar{y})}.
\]	
	If it is true that
	\begin{equation}\label{eq:Local-Stability-Condition}
	\rho_{0} = |T(\ell_{*},\ell_{*},r_{*})| < 1,
	\end{equation}
	then there exists $\delta > 0$ such that $\forall \, \ell^{(1)} \in [0,a]$ satisfying $|\ell^{(1)} - \ell_{*}| < \delta$, the sequence $\{\ell^{(n)}\}_{n \ge 1}$ in 
	\eqref{eq:GNEPit} converges to $\ell_{*}$. The convergence is exponential: for any $\varepsilon > 0$ there exists a positive constant $c(\ell^{(1)};\varepsilon)$ such that
	\begin{equation}\label{Proposition:Statement-Contraction-In-Neighbourhood-Of-Fixed-Point}
	|\ell^{(n)} - \ell_{*}| \le c(\ell^{(1)};\varepsilon)(\rho_{0} + \varepsilon)^{n}.
	\end{equation}
\end{proposition}

Our next result translates this into a local stability result for the DP. 
%The method is to show that under condition G1', the iteration \eqref{eq:Iterative-VF-Stopping} produces a sequence of strategies which are of threshold type. 

% we may therefore define stability of the DP as the convergence of the players' respective thresholds.  Proposition \eqref{Proposition:Contraction-In-Neighbourhood-Of-Fixed-Point} then yields 

\begin{corollary}\label{Theorem:Local-Stability-DP}
	Suppose Condition G1' holds and that $(D_{[0,\ell_{*}]},D_{[r_{*},1]})$ is a solution to the DP such that $(\ell_{*},r_{*}) \in (0,a) \times (b,1)$ and \eqref{eq:Local-Stability-Condition} holds. Then 
	%recalling the iteration in \eqref{eq:Iterative-VF-Stopping}, 
	%Let $\ell_{1} \in [0,a]$, $A_{1} = [0,\ell_{1}]$, and define two sequence of sets $\{A_{2n-1}\}_{n \ge 1}$ and $\{A_{2n}\}_{n \ge 1}$ inductively as follows,
 	%\begin{enumerate}[(a)]
 		%\item We have $A_{2n-1} = [0,\ell^{(n)}]$ and $A_{2n} = [r^{(n)},1]$, where $\ell^{(n)}$ and $r^{(n)}$ are given in \ref{eq:GNEPit} above with $\ell^{(1)} = \ell_{1}$.
		the equilibrium $(D_{[0,\ell_{*}]},D_{[r_{*},1]})$ in the DP is locally stable.
 	%\end{enumerate}
\end{corollary}

\begin{proof}

We have from Theorem \ref{Theorem:Necessary-and-Sufficiency-Optimality} that $(\ell_{*},r_{*}) \in (0,a) \times (b,1)$ is a solution to the GNEP. Applying Proposition \ref{Proposition:Contraction-In-Neighbourhood-Of-Fixed-Point}, take 
$\ell^{(1)} \in [0,a]$ satisfying $|\ell^{(1)} - \ell_{*}| < \delta$ and consider the  iteration given by \eqref{eq:GNEPit}. This yields sequences $(\ell^{(n)}) \to \ell_*$ and $(r^{(n)}) \to r_*$, taking values respectively in $(0,a)$ and $(b,1)$. Lemma~\ref{Lemma:Characterisation-1} and Theorem \ref{Theorem:Necessary-and-Sufficiency-Optimality} then show that the stopping time $D_{[r^{(n)},1]}$ is optimal in \eqref{eq:Iterative-VF-Stopping}-i). Similarly, the stopping time $D_{[0,\ell^{(n)}]}$ is optimal in \eqref{eq:Iterative-VF-Stopping}-iii).

Next we establish that the stopping region $A_{2}$ is given by $[r^{(1)},1]$. From Remark \ref{Remark:Necessary-Condition-For-Stopping}, we may study the optimal stopping problem \eqref{eq:Iterative-VF-Stopping}-i) in either of its equivalent forms \eqref{eq:Single-Payoff-OSP} or \eqref{eq:Single-Payoff-OSP-Alt-2-2} (taking $f=f_2$, $g=g_{2}$ and $A=A_1=[0,\ell^{(1)}$]). Using \eqref{eq:Single-Payoff-OSP}, it is immediate from the strict convexity of the obstacle $f_2$ on $[\ell^{(1)},b]$ and Dynkin's formula that $A_{2} \cap [\ell^{(1)},b] = \emptyset$. On the other hand, considering problem \eqref{eq:Single-Payoff-OSP-Alt-2-2} it follows from the strict concavity of the obstacle $f_2- g_{2,A_{1}}$ on $[b,1]$ and the smooth fit principle that the obstacle lies strictly below the value function on $[b,r^{(1)})$, establishing that $A_{2}=[r^{(1)},1]$. Arguing similarly for $A_3$ and then proceeding inductively we obtain $A_{2n+1} = [0,\ell^{(n+1)}]$ and $A_{2n+2} = [r^{(n+1)},1]$ for all $n$.
\end{proof}
\begin{remark}\label{rem:noint}
The fact that $A_1$ is an interval plays no role in the above proof, which only uses the inclusion $A_1 \subseteq [0,a]$. 
\end{remark}

%JM: It seemed possible, and natural, to bring the following result forward to here: 
{\bf Local stability in the zero-sum DP.}
The link to GNEPs also provides the following result on local stability of equilibria in the zero-sum DP, that is, when $f_{i} = -g_{-i}$, $i \in \{1,2\}$. The result is novel to the best of our knowledge.
\begin{theorem}\label{Theorem:Local-Stability-In-Zero-Sum-Game}
	Under Condition G1' every threshold-type solution of the zero-sum DP is locally stable.

	%, and that $(\ell_{*},r_{*}) \in (0,a) \times (b,1)$ is a solution to the GNEP. 
	%we have $T(\ell_{*},\ell_{*},r_{*}) = 0$ and the equilibrium $(\ell_{*},r_{*})$ is locally stable. \RMA{Consequently, the corresponding solution to the DP, 

\end{theorem}
\begin{proof}
Let a threshold-type solution $(D_{[0,\ell_{*}]},D_{[r_{*},1]})$ be given for the DP. We have $V_{1}^{[r_{*},1]} + V_{2}^{[0,\ell_{*}]} = 0$. Using the principle of smooth fit
%Lemma~\ref{Lemma:Smooth-Fit-At-Stop-First-Boundaries} and Theorem~\ref{Theorem:Nash-Equilibrium-Implicit-Functions}  
we get,
	\begin{align*}
	-g_{2}'(\ell_{*}) = f_{1}'(\ell_{*}) & = \frac{g_{1}(r_{*})-f_{1}(\ell_{*})}{r_{*} - \ell_{*}} \\
	& = \frac{[-f_{2}(r_{*}) + g_{2}(\ell_{*})]}{r_{*} - \ell_{*}}  = -f_{2}'(r_{*}) = g_{1}'(r_{*}).
	\end{align*}	
	Using the expressions for $U_{1}$ and $U_{2}$ in \eqref{eq:Auxiliary-GNZSG-Utilities}, the general expressions for the partial derivatives of the utility functions in Appendix~\ref{sec:smooth}, and the smooth fit principle at $(w,\bar{y})$ and $(\bar{x},\bar{y})$,
	%(cf. Lemma~\ref{Lemma:Smooth-Fit-At-Stop-First-Boundaries}), 
	one can show that
\begin{equation}\label{eq:Global-Stability-Condition-Operator}
T(w,\bar{x},\bar{y}) = \left(\frac{f_{1}'(\bar{x}) - g_{1}'(\bar{y})}{f_{1}''(\bar{x})(\bar{y}-\bar{x})}\right)\left(\frac{g_{2}'(w) - f_{2}'(\bar{y})}{f_{2}''(\bar{y})(\bar{y}-w)}\right).
\end{equation}
%Expression~\eqref{eq:Global-Stability-Condition-Operator} also appears in \cite{Attard2015} and \cite{DeAngelis2015} in their discussion on the uniqueness of Nash equilibria. However these references do not point out its importance for investigating the local stability of Nash equilibria.
In this zero-sum context we therefore have $T(\ell_{*},\ell_{*},r_{*}) = 0$, and the local stability of the equilibrium point now follows from Proposition~\ref{Proposition:Contraction-In-Neighbourhood-Of-Fixed-Point}.
\end{proof}

\subsection{Global stability and uniqueness}\label{sec:glob}
There is a stronger version of the criterion \eqref{eq:Local-Stability-Condition} that guarantees the iteration scheme to converge irrespective of player 1's initial strategy $\ell^{(1)} \in [0,a]$. Furthermore, the equilibrium strategy $(\ell_{*},r_{*})$ thus obtained is unique. 
\begin{theorem}\label{Theorem:Global-Stability-Theorem}
	Suppose that Condition G1' holds and that the reward functions $f_{i}$ and $g_{i}$, $i = 1,2$, satisfy
	\begin{equation}
		\sup_{w \in \mathcal{S}_{1}}\left|\left(\frac{f_{1}'(\bar{x}) - g_{1}'(\bar{y})}{f_{1}''(\bar{x})(\bar{y}-\bar{x})}\right)\left(\frac{g_{2}'(w) - f_{2}'(\bar{y})}{f_{2}''(\bar{y})(\bar{y}-w)}\right)\right| < 1,\label{eq:Global-Stability-Condition}
	\end{equation}
where $\bar{y} = \bar{y}(w)$ and $\bar{x} = \bar{x}(w)$ are defined by \eqref{eq:Best-Responses}. Then there exists $(\ell_{*},r_{*}) \in \mathcal{S}$ such that $(D_{[0,\ell_{*}]},D_{[r_{*},1]})$ is a solution to the DP. This solution is stable, and is unique in the class of Markovian strategies $(D_{S_1},D_{S_2})$ for closed stopping sets $S_1 \subseteq [0,a]$ and $S_2 \subseteq [b,1]$.
 % for the DP.
%JM: This definition of stability is now given above:
% in the following sense: if $A_{1} = [0,\ell^{(1)}]$ for any $\ell^{(1)} \in [0,a]$, then
%	\begin{equation}\label{eq:Convergence-Successive-Stopping-Sets}
%	\begin{split}
%	\liminf_{n \to \infty}A_{2n-1} & = \limsup_{n \to \infty}A_{2n-1} = [0,\ell_{*}]\\
%	\liminf_{n \to \infty}A_{2n} & = \limsup_{n \to \infty}A_{2n} = [r_{*},1]
%	\end{split}
%	\end{equation}
%	where $\{A_{2n-1}\}_{n \ge 1}$ and $\{A_{2n}\}_{n \ge 1}$ are given by \eqref{eq:Iterative-VF-Stopping}.
\end{theorem}
\begin{proof}
	Under Condition G1'
	%and \eqref{eq:Positive-Utility} hold then 
	every solution $(\ell_{*},r_{*})$ to the GNEP lies in $(0,a) \times (b,1)$. A standard contraction argument then shows that under \eqref{eq:Global-Stability-Condition}, there exists a unique solution $(\ell_{*},r_{*})$ to the GNEP and, further, that it is globally stable (see for example Theorem~1 in \cite{Li1987} or Proposition~4.1 in \cite{Basar1998}; see also Theorem~1.2.1 in \cite{Krasnoselskii1972}).
%gives a condition under which the mapping $S_{1} \colon \mathcal{S}_{1} \to \mathcal{S}_{1}$ discussed previously is a contraction, therefore guaranteeing that the successive approximations \eqref{eq:Successive-Approximations} converge to the {\em unique} fixed point of $S_{1}$.

Thus from Theorem \ref{Theorem:Necessary-and-Sufficiency-Optimality}, $(D_{[0,\ell_{*}]},D_{[r_{*},1]})$ is a solution to the DP. The fact that it is stable follows from the corresponding property in the GNEP.
Suppose that the DP has another solution $(D_{[0,\ell]},D_{[r,1]})$ with $\ell < r$. Again arguing as in Corollary~\ref{Theorem:Local-Stability-DP}, the reward function geometry 
%strict convexity of $f_{1}$ on $[a,1]$ and of $f_{2}$ on $[0,b]$ 
gives $\ell \in [0,a]$ and $r \in [b,1]$. Therefore $(\ell,r)$ is a solution to the GNEP 
%JM: This reference was already given in the proof of Corollary 5.2:
%(cf. Theorem~\ref{Theorem:Necessary-and-Sufficiency-Optimality}), 
and we have $\ell = \ell_{*}$ and $r = r_{*}$ by uniqueness. 

Suppose that $(D_{S_1},D_{S_2})$ is an equilibrium with closed stopping sets $S_1 \subseteq [0,a]$ and $S_2 \subseteq [b,1]$. Recalling Remark \ref{rem:noint}, now consider applying the iteration (ii) above, modified by choosing $A_1=S_1$, to obtain $A_2=[r,1]$, say. Then by optimality $S_2 \subseteq A_2$. 
Finally it is not difficult to see from a standard `small ball' argument that the strict concavity of $f_2$ on $[b,1]$ implies that $A_2 \setminus S_2 = \emptyset$. We conclude similarly that $A_1$ has the form $[0,\ell]$, completing the proof.
\end{proof}

\begin{remark} The sets $S_1$ and $S_2$ in Theorem \ref{Theorem:Global-Stability-Theorem} are closed in order to avoid trivialities, since every point is regular for standard Brownian motion. Note that the theorem establishes uniqueness among the Markovian strategies, rather than uniqueness among the subset of threshold-type strategies (cf. \cite{DeAngelis2015}). 
\end{remark}

\subsection{Examples}
We begin this section by constructing an example DP satisfying the global stability condition \eqref{eq:Global-Stability-Condition}. This example is then used to derive a second DP for which local stability, but not global stability, holds. Finally, we discuss local stability of the zero-sum DP.

{\bf Global stability.} Suppose that $b - a > \frac{1}{2}$ and that $F_i$, $G_i$ are functions satisfying Condition G1' and furthermore,
	%$g_1 \geq 0$ and that
	\begin{equation*}
	F_1(x) = x(\tfrac{a}{2}-x), \quad x \in [0,\tfrac{a}{2}].
	\end{equation*}
	It follows from Condition G1' that $F_1$ is negative on $[\tfrac{a}{2},1]$. Therefore, for every $w \in \mathcal{S}_{1}$ the `best response' $\bar x(w)$ to $\bar y(w)$ takes values in $[0,\tfrac{a}{2}]$, where we have the inequality
	\[
	\left|\frac{F_1'(x)}{F_1''(x)}\right| = \left|x - \tfrac{a}{4}\right| \leq \tfrac{1}{4}.
	\]
Since $G_1'$ is bounded on $[0,a]$ by Condition G1', and recalling that $\bar y \in [b,1]$ by definition, for a sufficiently large constant $R_1>0$ we have: 
	\[
%	\left|\frac{Kf_{1}'(\bar{x}) - g_{1}'(\bar{y})}{Kf_{1}''(\bar{x})(\bar{y}-\bar{x})}\right|
	\left|\frac{F_{1}'(\bar{x}) - \frac{1}{R_{1}} G_{1}'(\bar{y})}{F_{1}''(\bar{x})(\bar{y}-\bar{x})}\right| \le 2 \cdot \frac{1}{4} \cdot \frac{1}{b-a} < 1.
	\]
	Therefore if player 1's reward functions in the DP are $f_1 = F_1$ and $g_1=\frac{1}{R_1}G_1$ (which clearly satisfy Condition G1'), then the left hand parenthesis in \eqref{eq:Global-Stability-Condition} has absolute value less than 1. Similarly if we take $F_2(x)=(x-\frac{b+1}{2})(1-x)$ for all $x \in [\frac{b+1}{2},1]$ and let player 2's reward functions be $f_2 = F_2$ and $g_2=\frac{1}{R_2}G_2$ for a sufficiently large constant $R_2$, the right hand parenthesis in \eqref{eq:Global-Stability-Condition} has absolute value less than 1 and so the global stability condition \eqref{eq:Global-Stability-Condition} holds.
%\todo[inline]{JM: This is an interesting remark. Indeed this example violates Assumption 1 as $R_1 \to \infty$. Actually the {\em global} condition $f_{i} \le g_{i}$ seems rather strong. For example, do we really need $f_{1}(x) \le g_{1}(x)$ for $x \in [0,a]?$ If it is not admissible for player 2 to stop on $[0,a]$, then isn't the value of $g_1$ on $[0,a]$ completely irrelevant to the solution of the DP since this reward can never be received?
%
%RM: See Remark~\ref{Remark:Waiting-Incentive}.
%}	

	\begin{remark}
		Under Assumption~\ref{assumption:Preliminary-Assumption-1} the reward functions in the DP must satisfy $f_{i} \le g_{i}$ on $[0,1]$. Given the choice of $g_{i}$ in the example above, $f_{i} \le g_{i}$ implies that the rather strong condition $G_{i} \ge R_{i}F_{i}$ on $[0,1]$ must hold. Although Remark~\ref{Remark:Waiting-Incentive} shows that $G_{i} \ge R_{i}F_{i}$ is only needed on $\mathcal{S}_{-i}$, there are alternative choices for $g_{i}$ that satisfy Assumption~\ref{assumption:Preliminary-Assumption-1} and lead to a conclusion similar to that of the example above. More specifically, in the case $i = 1$, take any $G_{1} \ge \max(0,F_{1})$ which is in $C^{2}[0,1]$ and define $g_{1}$ to be a suitable restriction of $G_{1}$ to $[0,\frac{a}{2}]$ such that $g_{1}$ is in $C^{2}[0,1]$, and on $[b,1]$ $g_{1}$ is nonnegative and $g_{1}'$ is sufficiently small. For example, let $x \mapsto \eta(x)$ be the standard mollifier, $$\eta(x) = \begin{cases}
		C\exp\bigl(\tfrac{1}{x^{2}-1}\bigr),& |x| < 1 \\
		0,& |x| \ge 1
		\end{cases}$$ where $C > 0$ is chosen so that $\int_{\mathbb{R}}\eta(x){d}x=1$. For $\epsilon > 0$ define $\eta_{\epsilon}(x) \coloneqq \tfrac{1}{\epsilon}\eta(\tfrac{x}{\epsilon})$, $H_{\epsilon}(x) = \int_{-\infty}^{x}\eta_{\epsilon}(y){d}y$ and set $g_{1}(x;\epsilon) = H_{\epsilon}(\tfrac{a}{2}-x+\epsilon)G_{1}(x)$. For $x \le \tfrac{a}{2}$ we have $g_{1}(x;\epsilon) = G_{1}(x) \ge F_{1}(x) = f_{1}(x)$. For $x \ge \frac{a}{2} + 2\epsilon$ we have $g_{1}(x;\epsilon) = 0 \ge F_{1}(x) = f_{1}(x)$ and, for an appropriate choice of $\epsilon$, $g_{1}'(x) = 0$ on $[b,1]$.
	\end{remark}

%\todo[inline]{The purpose of the second example is to show that under Assumption \ref{Assumption:Smooth-Fit} (which is a strong assumption on the $f_i$), local and global stability are not the same condition.}

{\bf Local stability only.} Global stability implies that the local stability condition \eqref{eq:Local-Stability-Condition} holds at the unique Nash equilibrium $(\ell_*,r_*)$ in the DP we have just constructed. Taking the same reward functions in the DP, suppose now that player 1's strategy is $w_0 \in \mathcal S_1$ and that player 2's best response is $r_*$. Then from the smooth fit condition for player 2, the point $(w_0,g_2(w_0))$ must lie on the straight line tangent to $f_2$ at $(r_*, f_2(r_*))$. We may therefore conclude that if $g_2$ is not linear on $\mathcal S_1$, then there exists a strategy $w_0 \in \mathcal S_1 \setminus \{\ell_*\}$ for player 1 to which player 2's best response is $y_0 \in \mathcal S_2 \setminus \{r_*\}$. It is also not difficult to see that $y_0 \in (\frac{b+1}{2},1)$, and hence smooth fit holds at $y_0$, provided that $g_2$ is bounded above by the tangent to $f_2$ at $(1,f_2(1))$. 

Next we remark that the function $f_2$ may be arbitrarily `flattened' in a small neighbourhood of $ y_0$ without violating Condition G1'. That is, let $N_0$ be an open neighbourhood of $ y_0$ whose closure does not contain $r_*$ and let $\epsilon \in (f_2''( y_0),0)$. Then $f_2$ may be modified on $N_0$ to produce a new function $\tilde f_2$ with 
\begin{align*}
\tilde f_2(y) &= f_2(y),  \qquad y \in \{ y_0\} \cup N_0^c, \\
\tilde f_2'( y_0) &= f_2'( y_0), \\
\tilde f_2''( y_0) &= \epsilon, 
\end{align*}
and such that Condition G1' holds for the reward functions $f_1$, $\tilde f_2$ and $g_i$. By construction, the smooth fit condition 
%\eqref{eq:Player-2-Smooth-Fit-Equation} 
continues to hold at $y_0$ when $f_2$ is replaced by $\tilde f_2$, so that $y_0$ remains player 2's best response to $w_0$. In this way the right hand multiplicand in \eqref{eq:Global-Stability-Condition} may be made arbitrarily large in absolute value when $w=w_0$ (provided the numerator is non-zero, a mild condition). We thus obtain a DP satisfying Condition G1' which has local, but not global, stability.

\subsection{Uniqueness of Nash equilibria}\label{sec:glob2}
We close this section with a final result on uniqueness of equilibria in the DP by applying a well known condition in \cite{Rosen1965} for uniqueness of a solution to the GNEP.

\begin{theorem}\label{Theorem:Rosen-Uniqueness}
	Suppose that Condition G1' holds,
	\begin{align}\label{eq:Concavity-Utility-P1}
	f_{1}''(x) & \le -2\frac{f_{1}(x) + f_{1}'(x)(y-x) - g_{1}(y)}{(y-x)^{2}}, \quad \forall (x,y) \in (0,a) \times [b,1],\\
	f_{2}''(y) & \le -2\frac{f_{2}(y) - f_{2}'(y)(y-x) - g_{2}(x)}{(y-x)^{2}}, \quad \forall (x,y) \in [0,a] \times (b,1),\label{eq:Concavity-Utility-P2}
	\end{align}
	and $\exists\, (r_{1},r_{2}) \in [0,\infty) \times [0,\infty)$ such that $\forall (x,y) \in \mathcal{S}$,
	\begin{equation}\label{eq:Diagonal-Strict-Concavity-Condition-Explicit-1}
	4r_{1}r_{2}H_{1}(x,y)H_{2}(x,y) - \bigl(r_{1}H_{3}(x,y) + r_{2}H_{4}(x,y)\bigr)^{2} > 0,
	\end{equation}
	where $H_{1}$,\ldots,$H_{4}$ are given by,
	\begin{equation}\label{eq:Diagonal-Strict-Concavity-Condition-Explicit-2}
	\begin{split}
	H_{1}(x,y) & = f_{1}''(x)(y-x)^{2} + 2\bigl[f_{1}(x) + f_{1}'(x)(y-x) - g_{1}(y)\bigr] \\
	H_{2}(x,y) & = f_{2}''(y)(y-x)^{2} + 2\bigl[f_{2}(y) - f_{2}'(y)(y-x) - g_{2}(x)\bigr] \\
	H_{3}(x,y) & = 2\bigl[g_{1}(y) - f_{1}(x)\bigr] - (f_{1}'(x)+g_{1}'(y))(y-x) \\
	H_{4}(x,y) & = 2\bigl[g_{2}(x) - f_{2}(y)\bigr] + (g_{2}'(x) + f_{2}'(y))(y-x).
	\end{split}
	\end{equation}
	Then there exists a unique solution $(\ell_{*},r_{*}) \in \mathcal{S}$ to the GNEP~\eqref{eq:Auxiliary-GNEP}, and therefore $(D_{[0,\ell_{*}]},D_{[r_{*},1]})$ is the unique solution to the DP in the class of Markovian strategies $(D_{S_1},D_{S_2})$ for closed stopping sets $S_1 \subseteq [0,a]$ and $S_2 \subseteq [b,1]$.
\end{theorem}
\begin{proof}
	Conditions~\eqref{eq:Concavity-Utility-P1}--\eqref{eq:Concavity-Utility-P2} ensure that each utility function $s_{i} \mapsto U_{i}(s_{i},s_{-i})$, $i \in \{1,2\}$, is concave on $\mathcal{S}_{i}$ for each $s_{-i} \in \mathcal{S}_{-i}$. The condition~\eqref{eq:Diagonal-Strict-Concavity-Condition-Explicit-2} is sufficient for {\em strict diagonal concavity} according to Theorem 6 of \cite{Rosen1965}. The uniqueness result for the GNEP is an application of Theorem 2 in \cite{Rosen1965}, whereas uniqueness for the DP follows from the proof of Theorem~\ref{Theorem:Global-Stability-Theorem}.
\end{proof}
\begin{remark}\label{rem:xqc}
	For possible extensions of Theorem~\ref{Theorem:Rosen-Uniqueness} to quasi-concave utility functions see, for example, \cite{Arrow1961}. A comment on the relationship between the sufficient conditions for uniqueness of Nash equilibria used in Theorems~\ref{Theorem:Global-Stability-Theorem} and \ref{Theorem:Rosen-Uniqueness} can be found in Remark 3.3 of \cite{Li1987}.
\end{remark}

\section{Complex strategies and multiplayer GNEPs}\label{sec:examples}

The study of appropriate generalised games with $n>2$ players yields equilibria for the two-player Dynkin problem of Definition \ref{eq:NEP-Pure-Strategies} with more complex structures than the threshold type which has been previously studied. The systematic study of all cases $n>2$ is beyond the scope of this paper and so in this section we provide an example with $n=3$.

This example uses the following relaxation of Condition G1, under which the reward function $f_1$ has an additional convex portion:

%JM: Actually I would say the following generalisation is very interesting! However it would be a change of approach, so is better saved for another paper:
%It is not difficult to see that \eqref{eq:ext1} (resp. \eqref{eq:ext2}) does not require the convexity of $f_{1}$ (resp. $f_{2}$) on $[a,1]$ (resp. $[0,b]$). Therefore, in place of these convexity conditions one may just assume \eqref{eq:ext1}--\eqref{eq:ext2} hold for all $r \in [0,1]$ and $\ell \in [0,1]$ as appropriate. In fact, given the equivalent conditions proved in Theorem~\ref{Theorem:Necessary-and-Sufficiency-Optimality}, we can establish the existence of a solution to \eqref{eq:NEP-Pure-Strategies} with the a posteriori knowledge that \eqref{eq:ext1}--\eqref{eq:ext2} holds for a solution to \eqref{eq:Auxiliary-GNEP}. Rather than pursuing such generalisations, in this section we show how the previous results may be used to construct Nash equilibria in problems where at least one of the stopping strategies may not be of threshold type. In order to illustrate the idea, we impose the following set of assumptions on $f_{1}$ and $f_{2}$.
\smallskip
{\bf Condition G2.}
	There exist points $a_{1}$ and $a_{2}$ with $0 < a_{1} \le a_{2} < b < 1$ such that:
	\begin{align*}
	(i) \quad & f_{1} \text{ is convex on } [0,a_{1}], \text{ concave on } [a_{1},a_{2}] \text{ and convex on } [a_{2},1],\\
	(ii) \quad & f_{2} \text{ is convex on } [0,b] \text{ and concave on } [b,1].
	\end{align*}

Define sets $\hat{\mathcal{S}}_{1} = \hat{\mathcal{S}}_{2} = [a_{1},a_{2}]$,  $\hat{\mathcal{S}}_{3} = [b,1]$ and $\hat{\mathcal{S}} = \prod_{i=1}^{3}\hat{\mathcal{S}}_{i}$. Let the utility functions $\hat{U}_{i} \colon 
%\hat{\mathcal{S}} 
[0,1]^3 \to \bar{\mathbb{R}}$, $i \in \{1,2,3\}$ be defined by 
\begin{equation}\label{eq:Auxiliary-NZSG}
\begin{split}
\hat{U}_{1}(x,y,z) = \frac{f_{1}(x) - g_{1,[z,1]}(x)}{x}, \\
\hat{U}_{2}(x,y,z) = \frac{f_{1}(y) - g_{1,[z,1]}(y)}{z - y}, \\
\hat{U}_{3}(x,y,z) = \frac{f_{2}(z) - g_{2,[0,y]}(z)}{z - y},
\end{split}
\end{equation}
(taking $\hat{U}_{2}(x,y,z)=\hat{U}_{3}(x,y,z)=-\infty$ if $y \geq z$).
Define the players' feasible strategy spaces by the set-valued maps $\hat{K}_{i} \colon \hat{\mathcal{S}}_{-i} \rightrightarrows \hat{\mathcal{S}}_{i}$, where
\begin{equation}
\label{eq:3-Player-Feasible-Strategy-Spaces}
\hat{K}_{1}(y,z) = [a_{1},y \wedge a_{2}], \enskip \hat{K}_{2}(x,z) = [x \vee a_{1},a_{2}], \enskip \hat{K}_{3}(x,y) = [b,1],
\end{equation}
so that the feasible strategy triples belong to the convex, compact set $\hat{\mathcal{C}}$ defined by
	\begin{equation}\label{eq:constraintset-3-player}
	\hat{\mathcal{C}} = \{(x,y,z) \in [a_{1},a_{2}] \times [a_{1},a_{2}] \times [b,1] \colon x \le y\}.
	\end{equation}
%JM: This definition now appears earlier:
%For $s=(s_1,s_2,s_3) \in \hat{\mathcal{S}}$, $i \in \{1,2,3\}$ and $w \in \hat{\mathcal{S}}_i$ we will write $(s_{-i},w)$ for the strategy triple $s$ modified by replacing $s_i$ with $w$.
%we can assert the existence of a point $s^{*} \in \hat{\mathcal{S}}$ such that
%\begin{equation}\label{eq:Auxiliary-NEP-2}
%\hat{U}_{i}(s^{*}) = \sup\limits_{(s_{-i}^{*},s_{i}) \in \hat{\mathcal{C}}}\hat{U}_{i}((s_{-i}^{*},s_{i})), \enskip i \in \{1,2,3\}.
%\end{equation}

\begin{theorem}\label{Theorem:Existence-Nash-Equilibrium-Threshold-Only-One-Player}
	Suppose that the DP satisfies Condition G2. 
	Then:
	\begin{enumerate}
\item[(a)]
 there exists $s^{*} = (\ell^{1},\ell^{2},r) \in \hat{\mathcal{C}}$ with
\begin{equation}\label{eq:Auxiliary-NEP-2}
\hat{U}_{i}(s^{*}) = \sup\limits_{(s_{i},s_{-i}^{*}) \in \hat{\mathcal{C}}}\hat{U}_{i}(s_{i},s_{-i}^{*}), \enskip i \in \{1,2,3\},
\end{equation}
\item[(b)] a solution $s^{*} = (\ell^{1},\ell^{2},r) \in \hat{\mathcal C}$ to \eqref{eq:Auxiliary-NEP-2} satisfies 
	$\hat{U}_{2}(s^{*}) \geq 0$
	%$f_{1}(\ell^{2}) \ge g_{1,[r,1]}(\ell^{2})$ 
	if and only if $(D_{[\ell^{1},\ell^{2}]},D_{[r,1]})$ is a Nash equilibrium for the DP.
\end{enumerate}
\end{theorem}
\begin{proof}
Part (a) follows as in the proof of Lemma~\ref{Lemma:Existence-Solution-GNEP}. For part (b), we claim that the pair $(\ell^{1},\ell^{2})$ solves the following problem:
	\begin{quotation}
		\noindent{\bf Problem:} Find two points $\ell^{1}, \ell^{2}$ satisfying
		\begin{equation}\label{eq:Problem-P-}\tag{P}
		\begin{split}
		i) & \quad a_{1} \le \ell^{1} \le \ell^{2} \le a_{2}, \\
		ii) & \quad \hat{U}_{1}(x,\ell^{2},r)
%		\frac{f_{1}(x) - g_{1,[r,1]}(x)}{x} 
		\le 
		\hat{U}_{1}(\ell^{1},\ell^{2},r),
		%\frac{f_{1}(\ell^{1}) - g_{1,[r,1]}(\ell^{1})}{\ell^{1}}, 
		\quad \forall x \in (0,r), \\
		iii) & \quad \hat{U}_{2}(\ell^{1},y,r)
		%\quad \frac{f_{1}(x) - g_{1,[r,1]}(x)}{r-x} 
		\le \hat{U}_{2}(\ell^{1},\ell^{2},r),
		%\frac{f_{1}(\ell^{2}) - g_{1,[r,1]}(\ell^{2})}{r-\ell^{2}}, 
		\quad \forall y \in [0,r).
		\end{split}
		\end{equation}
%JM: The domain is now defined as [0,1]^3 above:			
%where $r \in \hat{\mathcal{S}}_{3}$ is player 3's strategy \re{in an equilibrium identified in part (a)}, and 
%			the domains of the functions $\hat{U}_{1}$ and $\hat{U}_{2}$ are extended using \eqref{eq:Auxiliary-NZSG} as appropriate.
	\end{quotation}
	%, using \eqref{eq:2-Maximum-Attained-On-Smaller-Set-2} above in the latter.

To establish part iii) note that the function $y \mapsto f_{1}(y) - g_{1,[r,1]}(y)$ is zero at $y=0$, convex for $y \in [0,a_{1}]$, concave for $y \in [a_{1},a_{2}]$, convex for $y \in [a_{2},r]$, nonnegative at $y=\ell^{2}$ and negative at $y=r$. It is then a straightforward exercise in convex analysis, similar to that in the proof of Theorem \ref{Theorem:Existence-Generalised-Nash-Equilibrium-Stopping}, to show that the maximum of the function $y \mapsto \hat{U}_{2}(\ell^{1},y,r)$ on $[0,r)$ must be attained at a point in $[a_1,a_2]$. Taking $i=2$ in \eqref{eq:Auxiliary-NEP-2} then establishes the claim. Part ii) follows similarly.

The necessity and sufficiency claim for the Nash equilibrium in stopping strategies then follows by applying Propositions~\ref{Proposition:Existence-Optimal-Strategy-Two-Points} and \ref{Proposition:Existence-Optimal-Strategy-Threshold-Only-One-Player} in the Appendix.
\end{proof}

%JM: I edited out the following comment. The first part seems like a general strategy for finding Nash equilibria in multiplayer games (find an equilibrium for two players then try to extend it to three players), rather than something specifically relevant to this study. But please shout if I've misunderstood. The second part I didn't quite understand -- all the conditions of Theorem 5.1 should be necessary, otherwise we wouldn't impose them. In the statement of the Theorem, this condition is now translated into a statement about nonnegative utility - perhaps this now seems more natural anyway?
\begin{comment}
\begin{remark}
The conclusion of Theorem~\ref{Theorem:Existence-Nash-Equilibrium-Threshold-Only-One-Player} can also be obtained by solving a $2$-player GNEP to determine the pair $(\ell^{2},r)$ first, then proving the existence of the point $\ell^{1} \le \ell^{2}$ that satisfies \eqref{eq:Proof-Nash-Point-Threshold-Only-One-Player-5}. Note also that the hypothesis $f_{1}(\ell^{2}) \ge g_{1,[r_{*},1]}(\ell^{2})$ is {\em necessary} due to the conditions of Problem~\eqref{eq:Problem-P-}. %This hypothesis is satisfied trivially If it is already known that for every $r \in [b,1]$ there exists a point $x_{r} \in [a_{1},a_{2}]$ such that $f_{1}(x_{r}) \ge g_{[r,1]}(x_{r})$.
\end{remark}
\end{comment}

\appendix
\numberwithin{equation}{section}

\section{Other Markov processes and discounting}\label{rem:ext}
Let $X = (X_{t})_{t \ge 0}$ be a continuous strong Markov process defined on an interval $E = (\ell,r)$. Suppose that the rewards in the DP are discounted by a factor $\lambda \ge 0$, so that \eqref{eq:Game-Payoff-Functional} becomes \begin{equation}\tag{1.1'}
\begin{split}
\mathcal{J}_{i}(\tau_{1},\tau_{2}) \coloneqq e^{-\lambda(\tau_{i} \wedge \tau_{-i})}\{f_{i}(X_{\tau_{i}})\mathds{1}_{\{\tau_{i} < \tau_{-i}\}} + g_{i}(X_{\tau_{-i}})\mathds{1}_{\{\tau_{-i} < \tau_{i}\}} + h_{i}(X_{\tau_{i}})\mathds{1}_{\{\tau_{i} = \tau_{-i}\}}\}, \\
i \in \{1,2\}.%, \quad j = 3-i.
\end{split}
\end{equation}
Lemma~\ref{Lemma:Characterisation-1} has a straightforward extension to the case $\lambda > 0$. Extending the concept of superharmonic functions in Definition~\ref{Definition:Superharmonic-Function}, we say that a measurable function $\phi \colon E_\Delta \to \mathbb{R}$ is {\em 
	$\lambda$-
	superharmonic} on a set $A \in \mathcal{B}(E_{\Delta})$ if for every $x \in E$ and $\tau \in \mathcal{T}$,
\[
\phi(x) \ge \mathds{E}^{x}[e^{-\lambda (\tau \wedge D_{A^{c}})}
\phi(X_{\tau \wedge D_{A^{c}}})].
\] The function $\phi_A$ introduced in Definition~\ref{def:red} is given more generally by,
\[
\phi_{A}(x) \coloneqq \mathds{E}^{x}\left[
e^{-\lambda D_{A}}\phi(X_{D_{A}})\right].
\]
It was noted in Section \ref{sec:fnclasses} that %the 
%$\lambda$-
%reduction
%function $x \mapsto \phi_A (x)$ defined in \eqref{eq:Lambda-Reduced-Function}
$\phi_A$ is continuous when $\lambda = 0$, $g$ is %bounded and measurable
continuous and $A$ is closed, since $X=X^E$ is a subprocess of a Brownian motion. This same property, which is important for ensuring that the obstacle in problem \eqref{eq:Single-Payoff-OSP-Alt-2-2} is continuous, also holds for $\lambda \ge 0$ when $X$ is a more general diffusion with strictly positive diffusion coefficient \cite{Schilling2012}. %is sufficient for the above results to hold with an appropriate modification of Condition G1. The property is:
%	\begin{description}\label{ass:cts}
%		\item[(R)]
%		\begin{quote}
%			For every continuous and bounded function $\phi \colon E \to \mathbb{R}$, every constant $\lambda \ge 0$, and every closed set $A \subseteq E$, the reduced function $x \mapsto \mathds{E}^{x}[e^{-\lambda D_{A}}\phi(X_{D_{A}})]$ is continuous.
%		\end{quote}
%	\end{description}
%\todo[inline]{JM: Do we require continuity or just measurability for the bounded function $\phi$? Currently we require the latter for the Brownian motion case and the former for general Markov processes}
Furthermore, when $X$ is a subprocess of a regular diffusion $Z = (Z_t)_{t \geq 0}$, the results in Sections \ref{Section:Existence-Of-Equilibria}--\ref{Section:Stability-Uniqueness} hold under an appropriate modification of Condition G1. We now briefly discuss this extension when $Z$ satisfies the stochastic differential equation,
\begin{equation}
{d}Z_{t} = \mu(Z_{t}){d}t + \sigma(Z_{t}){d}W_{t},
\end{equation}
where $W = (W_{t})_{t \ge 0}$ is a standard Brownian motion, $\mu \colon E \to \mathbb{R}$ and $\sigma \colon E \to (0,\infty)$ are measurable functions, $\mu$ bounded and $\sigma$ continuous, satisfying the following condition: for every $x \in E$,
\[
\int_{x-\varepsilon}^{x+\varepsilon}\frac{1 + |\mu(y)|}{\sigma^{2}(y)}{d}y < \infty \enskip \text{for some } \varepsilon > 0.
\]
Let $\mathcal{G} = \frac{1}{2}\sigma^{2}(\cdot)\frac{{{d}}^{2}}{{d}x} + \mu(\cdot)\frac{d}{{d}x}$ denote the infinitesimal generator corresponding to $Z$.

\subsection{Undiscounted rewards}
For the case $\lambda = 0$, we first recall from \cite{Dayanik2003} that there is a continuous increasing function $S$ on $E$, the {\em scale function}, which satisfies $\mathcal{G}S(\cdot) \equiv 0$. Let $\tilde{\ell} = S(\ell)$, $\tilde{r} = S(r)$ and $\tilde{X} = (\tilde{X}_{t})_{t \ge 0}$ be a Brownian motion on $\tilde{E} = (\tilde{\ell},\tilde{r})$. Then, it follows from Proposition~3.3 of \cite{Dayanik2003} that the DP corresponding to the process $X$ and rewards $f_i$, $g_{i}$ and $h_{i}$ on $E$ can be studied by an equivalent DP corresponding to $\tilde{X}$ with reward functions $\tilde{f}_{i}(\cdot) = f_i(S^{-1}(\cdot))$, $\tilde{g}_{i}(\cdot) = g_i(S^{-1}(\cdot))$, $\tilde{h}_{i}(\cdot) = h_i(S^{-1}(\cdot))$ on $\tilde{E}$.
\subsection{Discounted rewards}
For the case $\lambda > 0$, we first let $\psi^{\lambda}$ and $\phi^{\lambda}$ denote the fundamental solutions to the diffusion generator equation $\mathcal{G}w = \lambda w$, where $\psi^{\lambda}$ is strictly increasing and $\phi^{\lambda}$  is strictly decreasing \cite[p.~177]{Dayanik2003}. Let $F(\cdot) = \frac{\psi^{\lambda}(\cdot)}{\phi^{\lambda}(\cdot)}$, $\tilde{\ell} = F(\ell)$, $\tilde{r} = F(r)$ and $\tilde{X} = (\tilde{X}_{t})_{t \ge 0}$ be a Brownian motion on $\tilde{E} = (\tilde{\ell},\tilde{r})$. Then, it follows from Proposition~4.3 of \cite{Dayanik2003} that the DP corresponding to the process $X$ and rewards $f_i$, $g_{i}$ and $h_{i}$ on $E$ discounted by $\lambda > 0$ can be studied by an equivalent DP corresponding to $\tilde{X}$ with reward functions $\tilde{f}_{i}(\cdot) = \frac{f_i}{\phi^{\lambda}}(F^{-1}(\cdot))$, $\tilde{g}_{i}(\cdot) = \frac{g_i}{\phi^{\lambda}}(F^{-1}(\cdot))$, $\tilde{h}_{i}(\cdot) = \frac{h_i}{\phi^{\lambda}}(F^{-1}(\cdot))$ on $\tilde{E}$ {\em without discounting}.

\section{Expected payoffs for threshold strategies}\label{Section:Payoff-Derivation}
If players $1$ and $2$ use the strategies $D_{[0,\ell]}$ and $D_{[r,1]}$ respectively, where $0 \le \ell < r \le 1$, then the expected payoff $M^{x}_{1}(D_{[0,\ell]},D_{[r,1]})$ for player 1 (cf. \eqref{eq:Expected-Game-Payoff-Functional}) satisfies,

\begin{align*}
M^{x}_{1}(D_{[0,\ell]},D_{[r,1]}) = {} & \mathds{E}^{x}\bigl[f_{1}(X_{D_{[0,\ell]}})\mathds{1}_{\{D_{[0,\ell]} < D_{[r,1]}\}} + g_{1}(X_{D_{[r,1]}})\mathds{1}_{\{D_{[r,1]} < D_{[0,\ell]}\}}\bigr] \nonumber \\
& + \mathds{E}^{x}\bigl[h_{1}(X_{D_{[0,\ell]}})\mathds{1}_{\{D_{[0,\ell]} = D_{[r,1]}\}}\bigr] \nonumber \\
= {} & \begin{cases}
f_{1}(x),& \forall x \in [0,\ell] \\
f_{1}(\ell)\cdot\mathds{P}^{x}(\{D_{[0,\ell]} < D_{[r,1]}\}) + g_{1}(r)\cdot\mathds{P}^{x}(\{D_{[0,\ell]} > D_{[r,1]}\}),& \forall x \in (\ell,r) \\
g_{1}(x),& \forall x \in [r,1]
\end{cases} \nonumber \\
= {} & \begin{cases}
f_{1}(x),& \forall x \in [0,\ell] \\
f_{1}(\ell)\cdot\frac{r - x}{r - \ell} + g_{1}(r)\cdot\frac{x - \ell}{r - \ell},& \forall x \in (\ell,r) \\
g_{1}(x),& \forall x \in [r,1]
\end{cases}
\end{align*}
Analogously, the expected payoff $M^{x}_{2}(D_{[0,\ell]},D_{[r,1]})$ for player 2 satisfies,
\begin{equation*}
M^{x}_{2}(D_{[0,\ell]},D_{[r,1]}) = \begin{cases}
g_{2}(x),& \forall x \in [0,\ell] \\
g_{2}(\ell)\cdot\frac{r - x}{r - \ell} + f_{2}(r)\cdot\frac{x - \ell}{r - \ell},& \forall x \in (\ell,r) \\
f_{2}(x),& \forall x \in [r,1].
\end{cases}
\end{equation*}
\section{Derivatives of utility functions}\label{sec:smooth}
Throughout this section we assume Condition G1' holds. 
We first provide general formulas for the first and second partial derivatives of a utility function $U(x,y)$ which is of the form $U(x,y) = \frac{F(x,y)}{y-x}$.
\begin{align}
\partial_{x}U(x,y) & = \frac{\partial_{x}F(x,y)(y-x) + F(x,y)}{(y-x)^{2}}, \quad \partial_{y}U(x,y) = \frac{\partial_{y}F(x,y)(y-x) - F(x,y)}{(y-x)^{2}} \label{eq:Utility-Dx}\\
\partial_{xx}U(x,y) & = \frac{\partial_{xx}F(x,y)(y-x)^{2} + 2\bigl[\partial_{x}F(x,y)(y-x) + F(x,y)\bigr]}{(y-x)^{3}} \label{eq:Utility-Dxx}\\
\partial_{yy}U(x,y) & = \frac{\partial_{yy}F(x,y)(y-x)^{2} - 2\bigl[\partial_{y}F(x,y)(y-x) - F(x,y)\bigr]}{(y-x)^{3}}\label{eq:Utility-Dyy} \\
\partial_{xy}U(x,y) & = \frac{\partial_{xy}F(x,y)(y-x) + \partial_{x}F(x,y) + \partial_{y}F(x,y)}{(y-x)^{2}} - 2\frac{\bigl[\partial_{x}F(x,y)(y-x) + F(x,y)\bigr]}{(y-x)^{3}} \nonumber \\
& = \frac{\partial_{xy}F(x,y)(y-x) - \partial_{y}F(x,y) - \partial_{x}F(x,y)}{(y-x)^{2}} + 2\frac{\bigl[\partial_{y}F(x,y)(y-x) - F(x,y)\bigr]}{(y-x)^{3}} \label{eq:Utility-Dxy}
\end{align}
Using equation~\eqref{eq:Auxiliary-GNZSG-Utilities} for the utility functions gives the following expressions for their partial derivatives,
\begin{align*}
\partial_{x}U_{1}(x,y) & = \frac{f_{1}(x) + f_{1}'(x)(y-x) - g_{1}(y)}{(y-x)^{2}}, \quad \partial_{y}U_{2}(x,y) = \frac{g_{2}(x) + f_{2}'(y)(y - x) - f_{2}(y) }{(y-x)^{2}} \\
\partial_{xx}U_{1}(x,y) & = \frac{f_{1}''(x)(y-x)^{2} + 2\bigl[f_{1}(x) + f_{1}'(x)(y-x) - g_{1}(y)\bigr]}{(y-x)^{3}}\\
\partial_{yy}U_{2}(x,y) & = \frac{f_{2}''(y)(y-x)^{2} + 2\bigl[f_{2}(y) - f_{2}'(y)(y-x) - g_{2}(x)\bigr]}{(y-x)^{3}} \\
\partial_{xy}U_{1}(x,y) & = \frac{2\bigl[g_{1}(y) - f_{1}(x)\bigr] - (f_{1}'(x)+g_{1}'(y))(y-x)}{(y-x)^{3}}\\
\partial_{xy}U_{2}(x,y) & = \frac{2\bigl[g_{2}(x)- f_{2}(y)\bigr] + (g_{2}'(x) + f_{2}'(y))(y-x)}{(y-x)^{3}}
\end{align*}

\section{A verification theorem}\label{Section:Threshold-Extension-Example}
%JM: Problem (P) has already been stated above so the following restatement isn't necessary. Actually (ii) and (iii) had been permuted [I swapped them back], so it was previously quite confusing!!:
% be given and consider the following problem: 
%\begin{quotation}
%	{\bf Problem:} Find two points $\ell^{1}, \ell^{2}$ satisfying,
%	\begin{equation}\label{eq:Problem-P}\tag{P}
%	\begin{split}
%	i) & \quad a_{1} \le \ell^{1} \le \ell^{2} \le a_{2}, \\
%	ii) & \quad \frac{f_{1}(x) - g_{1,[r,1]}(x)}{x} \le \frac{f_{1}(\ell^{1}) - g_{1,[r,1]}(\ell^{1})}{\ell^{1}}, \quad \forall x \in (0,r), \\
%	iii) & \quad \frac{f_{1}(y) - g_{1,[r,1]}(x)}{r-y} \le \frac{f_{1}(\ell^{2}) - g_{1,[r,1]}(\ell^{2})}{r-\ell^{2}}, \quad \forall y \in [0,r),  
%	\end{split}
%	\end{equation}
%\end{quotation}
\begin{proposition}\label{Proposition:Existence-Optimal-Strategy-Two-Points}
	Under Condition G2 and given $r \in (a_{2},1]$, $(\ell^{1},\ell^{2})$ is a solution to Problem~\eqref{eq:Problem-P-} if and only if
	\begin{equation}\label{eq:Problem-P-Optimal-Stopping-1}
	V_{1}^{[r,1]}(x) \coloneqq \sup\limits_{\tau_{1} \in \mathcal{T}}M^{x}_{1}(\tau_{1},D_{[r,1]}) = M^{x}_{1}(D_{[\ell^{1},\ell^{2}]},D_{[r,1]}),\quad \forall x \in [0,1].
	\end{equation}
\end{proposition}
\begin{proof}
	The arguments are more or less the same as those establishing Theorem~\ref{Theorem:Necessary-and-Sufficiency-Optimality}. For the sake of brevity we therefore only show the proof of necessity (Problem~\eqref{eq:Problem-P-} $\implies$ \eqref{eq:Problem-P-Optimal-Stopping-1}). 
	
	Define $u_{r}$ on $[0,1]$ by,
	\begin{align}\label{eq:Candidate-Solution-Two-Point-Case}
	u_{r}(x) & = M^{x}_{1}(D_{[\ell^{1},\ell^{2}]},D_{[r,1]}) - g_{1,[r,1]}(x) \nonumber \\ 
	& = \begin{cases}
	\left(f_{1}(\ell^{1}) - g_{1,[r,1]}(\ell^{1})\right)\frac{x}{\ell^{1}},& x \in [0,\ell^{1}), \\
	f_{1}(x) - g_{1,[r,1]}(x),& x \in [\ell^{1},\ell^{2}), \\
	\left(f_{1}(\ell^{2}) - g_{1,[r,1]}(\ell^{2})\right)\frac{r-x}{r-\ell^{2}},& x \in [\ell^{2},r), \\
	0,& x \in [r,1].
	\end{cases}
	\end{align}
	
	Suppose $(\ell^{1},\ell^{2})$ is a solution to Problem~\eqref{eq:Problem-P-}. Similarly to Theorem~\ref{Theorem:Necessary-and-Sufficiency-Optimality}, we will prove \eqref{eq:Problem-P-Optimal-Stopping-1} by showing that $u_{r}$ is the smallest non-negative concave majorant of $f_{1} - g_{1,[r,1]}$ on $[0,r]$. Initially we will analyse $u_r$ separately on $[0,\ell^1]$ and $[\ell^1,\ell^2]$. %, so there is no loss of generality in assuming that $\ell^1>0$. 
	Observe firstly that the function $f_{1} - g_{1,[r,1]}$ is nonnegative when evaluated at the points $\ell^1$ and $\ell^2$ and hence, by concavity, on $[\ell^1,\ell^2]$. Recalling \eqref{eq:Auxiliary-NZSG}, this follows from  \eqref{eq:Problem-P-}, since $f_{1}(0) = g_{1,[r,1]}(0)$ and so $f_{1}(\ell^{2}) - g_{1,[r,1]}(\ell^{2}) \ge 0$. Also
	\[
	f_{1}(x) - g_{1,[r,1]}(x) \le \left(f_{1}(\ell^{1}) - g_{1,[r,1]}(\ell^{1})\right)\frac{x}{\ell^{1}}, \enskip \forall x \in (0,r),
	\]
and taking $x = \ell^{2}$ shows that $f_{1}(\ell^{1}) - g_{1,[r,1]}(\ell^{1}) \ge 0$. Therefore $u_{r}$ is a non-negative majorant of $f_{1} - g_{1,[r,1]}$ on $[0,\ell^{1}]$. This is also true on $[\ell^{1},r]$, since $f_{1}(r) \le g_{1}(r)$ and so
	\begin{equation}\label{eq:lastone}
		f_{1}(x) - g_{1,[r,1]}(x) \le \left(f_{1}(\ell^{2}) - g_{1,[r,1]}(\ell^{2})\right)\left(\frac{r-x}{r-\ell^{2}}\right),\enskip \forall x \in [0,r].
	\end{equation}	
Concavity holds for $u_{r}$ on the three intervals $[0,\ell^{1}]$, $[\ell^{1},\ell^{2}]$ and $[\ell^{2},r]$ separately and, arguing as in the proof of Theorem~\ref{Theorem:Necessary-and-Sufficiency-Optimality}, we can show that $u_{r}$ is continuous and concave on the entire interval $[0,r]$, completing the proof.
\end{proof}
\begin{proposition}\label{Proposition:Existence-Optimal-Strategy-Threshold-Only-One-Player}
	Under Condition G2, for every $\ell^{1},\ell^{2}$ satisfying $0 < \ell^{1} \le \ell^{2} < b$, a point $r \in [b,1]$ satisfies \eqref{eq:ex2} with $\ell = \ell^{2}$ and $U_2=\hat U_3$ if and only if
	\begin{equation}\label{eq:Threshold-Optimisation-2-2}
	V_{2}^{[\ell^{1},\ell^{2}]}(x) \coloneqq \sup\limits_{\tau_{2} \in \mathcal{T}}M^{x}_{2}(D_{[\ell^{1},\ell^{2}]},\tau_{2}) = M^{x}_{2}(D_{[\ell^{1},\ell^{2}]},D_{[r,1]}),\quad \forall x \in [0,1].
	\end{equation}
\end{proposition}
\begin{proof}
By Lemma \ref{Lemma:Characterisation-1} it is sufficient merely to consider the optimal stopping problem on the set $[0,\ell^1] \cup [\ell^2,1]$ with obstacle $f_{2} - g_{2,[\ell^1,\ell^2]}$, and we will only sketch the solution. Note that since $f_{2} \le g_{2}$ it is clearly 
suboptimal to stop in $[\ell^{1},\ell^{2}]$.
%Since $r$ does not belong to this interval, $D_{[r,1]}$ is an optimal stopping time when $x \in [\ell^{1},\ell^{2}]$.
%	\vskip 0.5em
%^	\noindent {\em Case} $2.$ $x \in [0,\ell^{1}]$
%	\vskip 0.5em
From Dynkin's formula it is also suboptimal to stop on $[0,\ell^{1}]$, since $f_{2} - g_{2,[\ell^{1},\ell^{2}]}$ is convex there and $f_{2}(x) - g_{2,[\ell^{1},\ell^{2}]}(x)\le 0$ for $x \in \{0,\ell^{1}\}$. %(recall \eqref{eq:gexp1}).
%on $[0,\ell^{1}]$ since,
%	\[
%	g_{2,[\ell^{1},\ell^{2}]}(x) = g_{2}(\ell^{1})\cdot\frac{x}{\ell^{1}},\quad \forall x \in [0,\ell^{1}].
%	\]
%	\begin{align*}
%	f_{2}(x) - g_{2,[\ell^{1},\ell^{2}]}(x) & \le \max(f_{2}(0) - g_{2,[\ell^{1},\ell^{2}]}(0),f_{2}(\ell^{1}) - g_{2,[\ell^{1},\ell^{2}]}(\ell^{1})) \\
%	& = \max(0,f_{2}(\ell^{1}) - g_{2}(\ell^{1})) \le 0.
%	\end{align*}
%	That is, $f_{2} \le g_{2,[\ell^{1},\ell^{2}]}$ on $[0,\ell^{1}]$. Recalling the discussion following Lemma~\ref{Lemma:Characterisation-1}, it is also optimal to continue in this case. Therefore $D_{[r,1]}$ is an optimal stopping time when $x \in [0,\ell^{1}]$.
The solution is nontrivial only on $(\ell^{2},1]$, where the arguments used for Theorem~\ref{Theorem:Necessary-and-Sufficiency-Optimality} are sufficient to complete the proof.
\end{proof}
%%%%%%%%%%%%%%%%%%%%%%%%%%%%%%%%%%%%%%%%%%%%%%%%%%%%%%%%%%%%%%%%%%%%%%%%%%%%%%%%%%%%%%%%%%%%%%%%%%%%%%%%%%%%%%%%%%%%%%%%%%%%%%%%%%%%%

%\bibliographystyle{hsiam} % outcomment this and next line in Case 1
%\bibliography{../../library,../../bibliography} % if more than one, comma separated

%\newpage
%\par \leftskip=24pt

%\noindent Randall Martyr \\
%School of Mathematical Sciences \\
%Queen Mary University of London \\
%Mile End Road \\
%London E1 4NS \\
%United Kingdom \\
%\texttt{r.martyr@qmul.ac.uk}
%\vspace{2pc}

%\par \leftskip=24pt

%\noindent John Moriarty \\
%School of Mathematical Sciences \\
%Queen Mary University of London \\
%Mile End Road \\
%London E1 4NS \\
%United Kingdom \\
%\texttt{j.moriarty@qmul.ac.uk}
\end{document}